\documentclass{mcrf}
\usepackage{amsmath}
\usepackage{paralist}
\usepackage{graphics} 	
\usepackage{epsfig}  		
\usepackage[colorlinks=true]{hyperref}
\hypersetup{urlcolor=blue, citecolor=red}

\def\currentvolume{2}

\def\ppages{X--XX}
\def\DOI{10.3934/mcrf.2012.2.xx}

\newcommand{\mat}[1]{\mbox{\boldmath{$#1$}}}
\newtheorem{theorem}{Theorem}[section]

\newtheorem{lemma}[theorem]{Lemma}
\newtheorem{proposition}{Proposition}

\theoremstyle{definition}

\newcommand{\jjntQ}{\jnt\!\!\!\!\jnt_{Q}}
\newcommand{\jjntq}{\jnt\!\!\!\!\jnt_{\mathcal{O}\times(0,T)}}

\def\dis{\displaystyle}
\def\Om{\Omega}
\def\om{\omega}
\def\yvec{\mathbf{y}}
\def\nvec{\mathbf{n}}
\def\xvec{\mathbf{x}}
\def\hvec{\mathbf{h}}
\def\evec{\mathbf{e}}
\def\vvec{\mathbf{v}}
\def\uvec{\mathbf{u}}
\def\varphii{\mat{\varphi}}
\def\Fvec{\mathbf{F}}
\def\0vec{\mathbf{0}}
\def\Hvec{\mathbf{H}}
\def\Evec{\mathbf{E}}
\def\Vvec{\mathbf{V}}
\def\Lvec{\mathbf{L}}
\def\Svec{\mathbf{S}}

\def\Gvec{\mathbf{G}}
\def\Wvec{\mathbf{W}}
\def\Cvec{\mathbf{C}}

\newcommand{\jnt}{\dis\int}

\newcommand{\N}{\mbox{$I \kern -4pt N$}}

\newcommand{\Q}{\mbox{$Q \kern -8pt I$}}
\newcommand{\C}{\mbox{$C \kern -8pt I$}}

\title[Controlling Boussinesq-like systems with few controls]
	{On the control of some coupled systems of the Boussinesq kind with few controls}

\author[Enrique Fern\'andez-Cara and Diego A. Souza]{}


\email{cara@us.es}
\email{diego@mat.ufpb.br}

\thanks{	The first author is supported by grants MTM2006-07932 and~MTM2010-15592 (DGI-MICINN, Spain).
        		The second author is supported by CAPES and CNPq Grant 620108/2008-8 (Brazil) and~MTM2010-15592 (DGI-MICINN, Spain).}

\begin{document}
\maketitle

\centerline{\scshape Enrique Fern\'andez-Cara}
\medskip
{\footnotesize
\centerline{Dpto.\ EDAN, University of Sevilla}
\centerline{Aptdo.~1160, 41080~Sevilla, Spain}
} 

\medskip

\centerline{\scshape Diego A. Souza}
\medskip
{\footnotesize
\centerline{Dpto.\ de Matem\'atica, Universidade Federal da Para\'iba}
\centerline{58051-900,~Jo\~ao Pessoa, Brazil}
}

\bigskip


\begin{abstract}
	This paper is devoted to prove the local exact controllability to the trajectories for a coupled system, of the 
	Boussinesq kind, with a reduced number of controls. In the state system, the unknowns are the velocity field and 
	pressure of the fluid $(\yvec,p)$, the temperature $\theta$ and an additional variable $c$ that can be viewed as 
	the concentration of a contaminant solute. We prove several results, that essentially show that it is sufficient 
	to act locally in space on the equations satisfied by $\theta$ and $c$.
\end{abstract}


\section{Introduction}

	Let $\Om\subset\mathbb{R}^N$ be a bounded connected open set whose boundary $\partial\Om$ is regular enough 
	(for instance of class $C^2$, $N=2$ or $N=3$). Let $\mathcal{O}\subset\Om$ be a (small) nonempty open subset 
	and assume that $T>0$. We will use the notation $Q=\Om\times(0,T)$ and $\Sigma=\partial\Om\times(0,T)$ and we 
	will denote by $\nvec=\nvec(\xvec)$ the outward unit normal to $\Om$ at any point $\xvec\in\partial\Om$.

	In the sequel, we will denote by $C$, $C_1$, $C_2$, $\dots$ various positive constants (usually depending on 
	$\Om$, $\mathcal{O}$ and $T$).

	We will be concerned with the following controlled system
\begin{equation}\label{eqsisacol}
	\left\{
		\begin{array}{lcl}
			\noalign{\smallskip} \dis
			\yvec_t - \Delta \yvec + (\yvec\cdot\nabla)\yvec + \nabla p = \vvec1_\mathcal{O} + \Fvec(\theta,c)	
			& \text{in}&	Q,		\\
			\noalign{\smallskip} \dis
			\nabla \cdot \yvec = 0                                                                                   						
			& \text{in}&     	Q,		\\
			\noalign{\smallskip} \dis
			\theta_t - \Delta \theta + \yvec \cdot \nabla\theta =  w_11_\mathcal{O} + f_1(\theta,c)                  		
			& \text{in}&     	Q,     		\\
			\noalign{\smallskip} \dis
			c_t - \Delta c +\yvec \cdot \nabla c = w_21_\mathcal{O} + f_2(\theta,c)                                  		
			& \text{in}&     	Q,     		\\
			\noalign{\smallskip} \dis
			\yvec = \0vec,~ \theta = c = 0                                                                           					
			& \text{on}&    	\Sigma, 	\\
     			\noalign{\smallskip} \dis
			\yvec(\cdot,0) = \yvec_0 , ~\theta(\cdot,0)=\theta_0, ~c(\cdot,0)=c_0                                                     				
			& \text{in}&    	\Omega,
		\end{array}
	\right.
\end{equation}
	where $\vvec = \vvec(\xvec, t)$, $w_1= w_1(\xvec, t)$ and $w_2 = w_2(\xvec, t)$ stand for the control functions. 
	They are assumed to act on the (small) subset $\mathcal{O}$ during the whole time interval $(0,T)$. The symbol 
	$1_\mathcal{O}$ stands for the characteristic function of $\mathcal{O}$. 
	
	It will be assumed that the functions $\Fvec = (F_1,\dots,F_N)$, $f_1$ and $f_2$ satisfy:
\begin{equation}\label{general-hyp-f}
	\left\{
		\begin{array}{l}
			F_i, f_1, f_2\in C^1(\mathbb{R}^2),~\hbox{with}~\nabla F_i,
			~\nabla f_1,~\nabla f_2\in \mathbf{L}^\infty(\mathbb{R}^2) \hbox{ and }  \\
			F_i(0,0) = f_1(0,0) = f_2(0,0) = 0~(1\leq i\leq N).
		\end{array}
	\right.
\end{equation}

	In~\eqref{eqsisacol}, $\yvec$ and $p$ can be, respectively, interpreted as the velocity field and the pressure of a fluid.
	The function $\theta$ (resp. $c$) can be viewed as the temperature of the fluid (resp. the concentration of a contaminant solute).
	On the other hand, $\vvec$, $w_1$ and $w_2$ must be regarded as source terms, locally supported in space, respectively for 
	the PDEs satisfied by $(\yvec,p)$, $\theta$ and $c$.

	From the viewpoint of control theory, $(\vvec,w_1,w_2)$ is the control and $(\yvec,p,\theta,c)$ is the state.
	In the problems considered in this paper, the main goal will always be related to choose $(\vvec,w_1,w_2)$ such that 
	$(\yvec,p,\theta,c)$ satisfies a desired property at $t=T$.

	More precisely, we will present some results that show that the system \eqref{eqsisacol} can be controlled, at least locally, 
	with only $N$ scalar controls in $L^2(\mathcal{O}\times(0,T))$. We will also see that, when $N=3$, \eqref{eqsisacol} can be 
	controlled, at least under some geometrical assumptions, with only $2$ (i.e., $N-1$) scalar controls.
	
	Thus, let us introduce the spaces $\Hvec$, $\Evec$ and $\Vvec$, with
\[
	\begin{array}{ll}
		&\Hvec = \{\,\varphi \in \Lvec^2(\Om) : \nabla \cdot \varphi = 0~\hbox{in}~\Om~\hbox{and}
		~\varphi\cdot\nvec = 0~\hbox{on}~\partial \Om\,\},                                               				\\ 
		\noalign{\smallskip}\dis
		&\Vvec = \{\,\varphi\in\Hvec^1_0(\Om) : \nabla \cdot \varphi = 0~\hbox{in}~\Om\,\},		\\
		\noalign{\smallskip}\dis 
		&\Evec = 
			\left\{
			\begin{array}{lcl}
				\Hvec,  				&\hbox{if}& N = 2,                                                                         	\\
				\Lvec^4(\Om) \cap \Hvec , 	&\hbox{if}& N = 3
			\end{array}
			\right.
	\end{array}
\]
	and let us fix a {\it trajectory} $(\overline{\yvec},\overline{p},\overline{\theta},\overline{c})$, that is, 
	a sufficiently regular solution to the related noncontrolled system:
\begin{equation}\label{trajectory}
	\left\{
		\begin{array}{lcl}
			\overline{\yvec}_t - \Delta \overline{\yvec} + (\overline{\yvec} \cdot \nabla) \overline{\yvec}
			+ \nabla \overline{p} = \Fvec(\overline{\theta},\overline{c})				&\hbox{in}	&	Q,		\\
			\noalign{\smallskip}\dis
			\nabla \cdot \overline{\yvec} = 0                                  						&\hbox{in}	& 	Q,		\\
			\noalign{\smallskip}\dis
			\overline{\theta}_t - \Delta \overline{\theta} +\overline{\yvec} \cdot \nabla\overline{\theta}
			= f_1(\overline{\theta},\overline{c})                            						&\hbox{in}	& 	Q,		\\
			\noalign{\smallskip}\dis
			\overline{c}_t - \Delta \overline{c} +\overline{\yvec} \cdot \nabla \overline{c} 
			= f_2(\overline{\theta},\overline{c})                                                                 	 	&\hbox{in}	& 	Q,		\\
			\noalign{\smallskip}\dis
			\overline{\yvec} = \0vec, ~\overline{\theta}=\overline{c}=0                               	&\hbox{on}&	\Sigma,	\\
			\noalign{\smallskip}\dis
			\overline{\yvec}(\cdot,0) = \overline{\yvec}_0 , 
			~\overline{\theta}(\cdot,0) = \overline{\theta}_0, ~\overline{c}(\cdot,0) = \overline{c}_0	&\hbox{in}	&	\Om.
		\end{array}
	\right.
\end{equation}
	It will be assumed that
\begin{equation}\label{trajec}
	\overline{\yvec}_i,~\overline{\theta},~\overline{c}\in L^\infty(Q)~\hbox{and}~\overline{\yvec}_{i,t},
	~\overline{\theta}_t,\overline{c} _t\in L^2(0,T;L^\kappa(\Om)),~(1\leq i\leq N)
\end{equation}
	with
\begin{equation}\label{kappa}	
		\left\{
			\begin{array}{lcl}
				\kappa >1 ,   &\hbox{if}& N = 2, \\
				\kappa >6/5 , &\hbox{if}& N = 3.
			\end{array}
		\right.
\end{equation}

	Notice that, if the initial data in \eqref{trajectory} satisfy appropriate regularity conditions and 
	$(\overline{\yvec},\overline{p},\overline{\theta},\overline{c})$ solves \eqref{trajectory} (for instance 
	in the usual weak sense) and $\overline{\yvec}_i$, $\overline{\theta}$, $\overline{c}\in L^{\infty}(Q)$, then 
	we have \eqref{trajec}. For example, if $\overline{\yvec}_0\in\Vvec$ and $\overline{\theta}_0$, 
	$\overline{c}_0 \in H_0^1(\Om)$, we actually have from the parabolic regularity theory that 
	$\overline{\yvec}_{i,t}$, $\overline{\theta}_t$, $\overline{c} _t \in L^2(Q)$.

	In our first main result, we will assume the following:
\begin{equation}\label{f-case-1}
	\begin{array}{l}
		\dis f_1\equiv f_2\equiv0~\hbox{and}~\Fvec(a_1,a_2)=a_1\evec_N+a_2\vec{\hvec},~\hbox{where:}	\\
		\noalign{\medskip}
		\dis~~\bullet~\hbox{$\evec_N$ is the $N$-th vector of the canonical basis of $\mathbb{R}^N$,}			\\
		\noalign{\smallskip}
		\dis~~\bullet~\hbox{$\vec{\hvec}$ is a vector of $\mathbb{R}^N$ such that $\{\evec_N,\vec{\hvec}\}$ is linearly independent.}
	\end{array}
\end{equation}

	Then, we have the following result:
\begin{theorem}\label{teo1}
	Assume that the assumptions \eqref{trajectory}--\eqref{f-case-1} are satisfied. Then there exists 
	$\delta > 0$ such that, whenever $(\yvec_0,\theta_0,c_0)\in\Evec\times L^2(\Om)\times L^2(\Om)$ and
\[
	\|(\yvec_0, \theta_0, c_0)-(\overline{\yvec}_0,\overline{\theta}_0,\overline{c}_0)\|\leq\delta,
\]
	we can find a $L^2$-control $(\vvec, w_1,w_2)$ with $v_i\equiv v_N\equiv0$ for some $1\leq i<N$ 
	and an associated state $ (\yvec,p,\theta,c)$ satisfying
\begin{equation}\label{contrage}
	\yvec(\cdot,T) = \overline{\yvec} (\cdot,T),~\theta(\cdot,T) = \overline{\theta}(\cdot,T)~\hbox{and}~c(\cdot,T) = \overline{c}(\cdot,T)\quad \hbox{in}\quad \Omega.
\end{equation}
\end{theorem}

	In our second result, we will consider more general (and maybe nonlinear) functions $\Fvec$. We will denote by 
	$\mathbf{G}(\theta, c)$ and $\mathbf{L}(\theta, c)$ the partial derivatives of $\Fvec$ with respect to $\theta$ and $c$:
$$
	\mathbf{G}(\theta, c) = {\partial \Fvec \over \partial \theta}(\theta, c)\quad\hbox{and}\quad\mathbf{L} (\theta, c)= {\partial \Fvec \over \partial c}(\theta, c).
$$
	The following will be assumed:
\begin{equation}\label{f-case-2}
	\begin{array}{l}
		\dis\hbox{There exists a non-empty open subset}~\mathcal{O}_*\subset\mathcal{O}~\hbox{such that}	\\
		\noalign{\smallskip}\dis 
		\mathbf{G}(\overline{\theta},\overline{c})~\hbox{and}~\mathbf{L}(\overline{\theta},\overline{c})
		~\hbox{are continuous and linearly independent in}~\mathcal{O}_* \times (0,T).
	\end{array}
\end{equation}

	Then, we get a generalization of Theorem \ref{teo1}:
\begin{theorem}\label{teo2}
	Assume that the assumptions \eqref{general-hyp-f}, \eqref{trajectory}--\eqref{kappa} and \eqref{f-case-2} are satisfied.
	Then there exists $\delta > 0$ such that, whenever $(\yvec_0,\theta_0,c_0)\in \Evec\times L^2(\Om)\times L^2(\Om)$ 
	and
$$
	\|(\yvec_0,\theta_0,c_0) - (\overline{\yvec}_0,\overline{\theta}_0,\overline{c}_0)\| \leq \delta,
$$
	we can find a $L^2$-control $(\vvec, w_1,w_2)$ with $v_i\equiv v_j\equiv0$ for some $i\neq j$ and an associated 
	state $ (\yvec,p,\theta,c)$ satisfying \eqref{contrage}.
\end{theorem}

	In the three-dimensional case, we can improve Theorem \ref{teo1} if we add to the hypotheses an appropriate geometrical 
	assumption on $\mathcal{O}$. More precisely, let us assume that:
\begin{equation}\label{controldomain}
	\begin{array}{l}
	\hbox{there exist}~\xvec^0\in\partial\Om~\hbox{and}~a>0~\hbox{such that}~B_{a}(\xvec^0)
	\cap\partial\Om\subset\overline{\mathcal{O}}\cap\partial\Om,\\
	\noalign{\smallskip}\dis 
	\hbox{where $B_{a}(\xvec^0)$ is the ball centered at $\xvec^0$ of radius $a$.}
	\end{array}
\end{equation}

\
	
	\noindent Then the following holds:
\begin{theorem}\label{teo3}
	Assume that $N=3$, the assumptions in Theorem \ref{teo1} are satisfied, \eqref{controldomain} holds and
\begin{equation}\label{controldomain-bis}
	h_1n_2(\xvec^0)-h_2n_1(\xvec^0)\neq0.
\end{equation}
	Then, the conclusion of Theorem \ref{teo1} holds true with a $L^2$-control $(\vvec, w_1,w_2)$ such that 
	$\vvec\equiv\0vec$.
\end{theorem}

	The rest of this paper is organized as follows. In Section \ref{Sec2}, we recall a previous result, needed for the proofs 
	of Theorems \ref{teo1} to \ref{teo3}. In Section \ref{Sec3}, we give the proof of Theorem \ref{teo1}. We will adapt the 
	arguments in \cite{FC-G-P} and \cite{elim.control}, that lead to the local exact controllability to the trajectories for 
	Navier-Stokes and Boussinesq systems; see also \cite{F-I3, Guerrero, Imanuvilov1}. It will be seen that the main ingredients 
	of this proof are appropriate global Carleman estimates for the solutions to linear systems similar to \eqref{eqsisacol} and 
	an inverse mapping theorem of the Liusternik's kind. Sections \ref{Sec4} and \ref{Sec5} respectively deal with the proofs of 
	Theorems \ref{teo2} and \ref{teo3}. In Section \ref{Sec6}, we present some additional questions and comments. Finally, for 
	completeness, we recall the main ideas of the proof of the Carleman estimates that serve as a starting point in an Appendix
	(see Section \ref{Appendix}).


\section{A preliminary result}\label{Sec2}
	
	A considerable part of this paper follows from the arguments and results in \cite{FC-G-P} and \cite{elim.control} adapted 
	to the present context. Thus, let us set $\yvec=\overline{\yvec}+\uvec$, $p=\overline{p}+q$, $\theta=\overline{\theta}+\phi$,
	$c=\overline{c}+z$ and let us use these identities in \eqref{eqsisacol}. Taking into account that 
	$(\overline{\yvec},\overline{p},\overline{\theta},\overline{c})$ solves \eqref{trajectory}, we find:
\begin{equation}\label{sisnullcont}
	\hspace{-0.3cm}\left\{\!\!\!
		\begin{array}{lcl}
			\uvec_t - \Delta \uvec \!+\! (\uvec\cdot\nabla)\overline{\yvec} + (\overline{\yvec}\cdot\nabla)\uvec
			 + (\uvec\cdot \nabla)\uvec + \nabla q = \vvec1_\mathcal{O} + \Fvec(\theta,c)-\overline{\Fvec}             	
			 												\!\!\!\!&\hbox{in}&\!\!\!\!  	Q,		\\
			\noalign{\smallskip}\dis 
			\nabla \cdot \uvec = 0                                                                                     	\!\!\!\!&\hbox{in}&\!\!\!\!	Q,      		\\
			\noalign{\smallskip}\dis 
			\phi_t - \Delta \phi +  \uvec\cdot\nabla\overline{\theta}+\overline{\yvec}\cdot\nabla\phi 
			+\uvec\cdot\nabla \phi= w_11_\mathcal{O}+ f_1(\theta,c)-\overline{f}_1                                     			
															\!\!\!\!&\hbox{in}&\!\!\!\!   	Q,      		\\
			\noalign{\smallskip}\dis 
			z_t - \Delta z + \uvec \cdot \nabla \overline{c}+\overline{\yvec}\cdot\nabla z
			+ \uvec \cdot \nabla z=  w_21_\mathcal{O}+f_2(\theta,c)-\overline{f}_2                                     			
															\!\!\!\!&\hbox{in}&\!\!\!\!   	Q,      		\\
			\noalign{\smallskip}\dis 
			\uvec = \0vec,~\phi = z = 0                                                                               	\!\!\!\!&\hbox{on}&\!\!\!\!  	\Sigma,	\\
			\noalign{\smallskip}\dis 
			\uvec(\cdot,0) = \yvec_0-\overline{\yvec}_0 ,~\phi(\cdot,0)=\theta_{0}-\overline{\theta}_0, ~z(\cdot,0)=c_0-\overline{c}_0  	
															\!\!\!\!&\hbox{in}&\!\!\!\!   	\Omega,
		\end{array}
	\right.
\end{equation}
	where we have introduced $\overline{\Fvec} := \Fvec(\overline{\theta},\overline{c})$,$~\overline{f}_1 := f_1(\overline{\theta},\overline{c})$ and 
	$\overline{f}_2 := f_2(\overline{\theta},\overline{c})$.

	This way, the local exact controllability to the trajectories for the system \eqref{eqsisacol} is reduced to a local null 
	controllability problem for the solution $(\uvec,q,\phi,z)$ to the nonlinear problem \eqref{sisnullcont}.

	In order to solve the latter, following a standard approach, we will first deduce the (global) null controllability of 
	a suitable linearized version, namely:
\begin{equation}\label{sislinear}
	\left\{
		\begin{array}{lcl}
			\noalign{\smallskip}\dis 
			\uvec_t  - \Delta \uvec + (\uvec\cdot\nabla)\overline{\yvec} + (\overline{\yvec}\cdot\nabla)\uvec 
			+ \nabla q = \mathbf{S} + \vvec1_\mathcal{O} + \overline{\Gvec}\phi + \overline{\Lvec}z                 		
																							& \hbox{in} & Q,     \\
			\noalign{\smallskip}\dis 
			\nabla \cdot \uvec = 0                                                                            							& \hbox{in} & Q,     \\
			\noalign{\smallskip}\dis 
			\phi_t - \Delta \phi + \uvec\cdot\nabla\overline{\theta} + \overline{\yvec}\cdot\nabla\phi
			= r_1 + w_11_\mathcal{O} + \overline{g}_1\phi + \overline{l}_1z                                                    		& \hbox{in} & Q,     \\
			\noalign{\smallskip}\dis 
			z_t - \Delta z + \uvec \cdot \nabla \overline{c} + \overline{\yvec}\cdot\nabla z
			= r_2 + w_21_\mathcal{O} + \overline{g}_2\phi + \overline{l}_2z                                                    		& \hbox{in} & Q,     \\
			\noalign{\smallskip}\dis 
			\uvec = \0vec, ~\phi = z = 0                                                                                       					& \hbox{on} & \Sigma,\\
			\noalign{\smallskip}\dis 
			\uvec(\cdot,0) = \uvec_0= , ~\phi(\cdot,0)=\phi_{0}, ~z(\cdot,0)=z_0                                                                  			& \hbox{in} & \Omega,
		\end{array}
	\right.
\end{equation}
	where $ \uvec_0:= \yvec_0-\overline{\yvec}_0$, ~$\phi_{0}:=\theta_{0}-\overline{\theta}_0$,  $z_0:=c_0-\overline{c}_0$,
	$\overline{\Gvec} := \Gvec(\overline{\theta},\overline{c})$, $\overline{\Lvec} := \Lvec(\overline{\theta},\overline{c})$,
	$\overline{g}_i := g_i(\overline{\theta},\overline{c})$, $\overline{l}_i := l_i(\overline{\theta},\overline{c})$,
 	$g_i$ and $l_i$ denote the partial derivatives of $f_i$ with respect to $\theta$ and $c$, respectively, for $i=1,2$,
	and $\Svec$, $r_1$ and $r_2$ are appropriate functions that decay exponentially as $t\rightarrow T^-$.
	Then, appropriate and rather classical arguments will be used to deduce the local null controllability of the nonlinear 
	system \eqref{sisnullcont}.

	In this Section, we will present a suitable Carleman inequality for the so called adjoint of \eqref{sislinear}. This will 
	lead easily to the null controllability result.

	Thus, let us first introduce some weight functions:
\[
	\begin{alignedat}{2}
		\noalign{\smallskip}\dis 
		\alpha(x,t) &= \frac{e^{5/4\lambda m \|\eta^0\|_{\infty}} - e^{\lambda(m\|\eta^0\|_{\infty}+\eta^0(x))}}{t^4(T-t)^4},
		\quad
		\xi(x,t) = \frac{e^{\lambda(m\|\eta^0\|_{\infty}+\eta^0(x))}}{t^4(T-t)^4},                                      
		\\
		\noalign{\smallskip}\dis 
		\widehat{\alpha}(t)&	= \dis \min_{x\in\overline{\Omega}}\alpha(x,t) 
						=\frac{e^{5/4\lambda m \|\eta^0\|_{\infty}} - e^{\lambda(m+1)\|\eta^0\|_{\infty}}}{t^4(T-t)^4},		\\
		\noalign{\smallskip}\dis 
		\alpha^*(t)&	= \dis \max_{x\in\overline{\Omega}}\alpha(x,t) 
		                     		=  \frac{e^{5/4\lambda m \|\eta^0\|_{\infty}} - e^{\lambda m\|\eta^0\|_{\infty}}}{t^4(T-t)^4},            
		\\
		\noalign{\smallskip}\dis 
		\widehat{\xi}(t)&	= \dis \min_{x\in\overline{\Omega}}\xi(x,t)
		                     			=  \frac{e^{\lambda m\|\eta^0\|_{\infty}}}{t^4(T-t)^4},                                                    		
		 \quad
		 \xi^*(t)    = \dis \max_{x\in\overline{\Omega}}\xi(x,t)
		                     =  \frac{e^{\lambda (m+1)\|\eta^0\|_{\infty}}}{t^4(T-t)^4},                                                				
		\\
		\noalign{\smallskip}\dis 
		\widehat{\mu}(t)&	= s\lambda e^{-s\widehat{\alpha}}\xi^*,
		\mu(t)	=  s^{15/4}e^{-2s\widehat{\alpha}+s\alpha^*}{\xi^*}^{15/4},
	\end{alignedat}
\]
	where $m>4$ is a fixed real number, $\eta^0\in C^2(\overline{\Om})$ is a function that verifies
\begin{center}
	$\eta^0>0$ in $\Om$, $\eta^0=0$ on $\partial\Om$ and $|\nabla\eta^0|>0 $ in 
	$\overline{\Om}\setminus \mathcal{O}_0$
\end{center}
	and $\mathcal{O}_0$ is a non-empty open subset of $\mathcal{O}$ such that $\overline{\mathcal{O}}_0\subset \mathcal{O}$.

	The adjoint system of \eqref{sislinear} is:
\begin{equation}\label{adjoint}
	\left\{
		\begin{array}{lcl}
			\noalign{\smallskip}\dis
			-\varphii_t - \Delta\varphii - D\varphii\overline{\yvec} + \nabla \pi = \widetilde{\Gvec}
			+\overline{\theta}\nabla\psi +  \overline{c}\nabla\zeta                        			&\hbox{in}& Q,      	\\
			\noalign{\smallskip}\dis
			\nabla \cdot\varphii = 0                                                                           				&\hbox{in}& Q,      	\\
		    	\noalign{\smallskip}\dis
			-\psi_t - \Delta \psi -\overline{\yvec} \cdot \nabla\psi =  \widetilde{g}_1
			+\overline{\Gvec}\cdot\varphii+ \overline{g}_1\psi+\overline{g}_2\zeta           	&\hbox{in}& Q,      	\\
		    	\noalign{\smallskip}\dis
			-\zeta_t - \Delta \zeta -\overline{\yvec} \cdot \nabla\zeta = \widetilde{g}_2
			+\overline{\Lvec}\cdot\varphii+ \overline{l}_1\psi+\overline{l}_2\zeta            	&\hbox{in}& Q,      	\\
			\noalign{\smallskip}\dis
			\varphii = \0vec,~\psi=\zeta=0                                                                    			&\hbox{on}& \Sigma, 	\\
			\noalign{\smallskip}\dis
			\varphii(\cdot,T) = \varphii_T ,~\psi(\cdot,T)=\psi_T, ~\zeta(\cdot,T)=\zeta_T                         	&\hbox{in}& \Omega,
		\end{array}
	\right.
\end{equation}
	where $D \varphii = \nabla \varphii + \nabla \varphii^{tr}$ denotes the symmetric part of the gradient of $\varphii$. 
	Here, the final and right hand side data are assumed to satisfy:
\[
	\varphii_T\in \Hvec,~\psi_T,~\zeta_T\in L^2(\Om),~\widetilde{G}_i,~\widetilde{g}_1,\widetilde{g}_2\in L^2(Q)
	~(1\leq i\leq N).
\]

	Let us introduce the following notation:
\[
	\begin{alignedat}{2} 
		\noalign{\smallskip}
		I(s,\lambda;g) = ~&\dis
		s^{-1}\jjntQ \xi^{-1}e^{-2s\alpha}|g_t|^2\,dx\,dt + s^{-1}\jjntQ \xi^{-1}e^{-2s\alpha}|\Delta g|^2 \,dx\,dt\\ 
	    	\noalign{\smallskip}
		&\dis	+\,s\lambda^2\jjntQ \xi e^{-2s\alpha}|\nabla g|^2 \,dx\,dt + s^3\lambda^4\jjntQ \xi^3e^{-2s\alpha}|g|^2 \,dx\,dt
	\end{alignedat}
\]
	for any $s, \lambda > 0$ and for any function $g = g(x,t)$ such that these integrals of $g$ make sense. Let us also set
\[
	K(\varphii, \psi,\zeta)=I(s,\lambda;\varphii)+I(s,\lambda;\psi)+I(s,\lambda;\zeta).
\]

	For the moment, we will accept the following proposition, whose proof is sketched in the Appendix:
\begin{proposition}\label{Carleman}
	Assume that $(\overline{\yvec},\overline{p},\overline{\theta},\overline{c})$ satisfies \eqref{trajectory}--\eqref{kappa}.
	There exist positive constants $\widehat{s}$, $\widehat{\lambda}$ and $\widehat{C}$, only depending on $\Om$ and 
	$\mathcal{O}$ such that, 
	for any $(\varphii _T,\psi_T,\zeta_T)\in\Hvec\times L^2(\Om)\times L^2(\Om) $ and any 
	$(\widetilde{\Gvec},\widetilde{g}_1,\widetilde{g}_2) \in \Lvec^2(Q)\times L^2(Q)\times L^2(Q)$, the solution to the adjoint 
	system \eqref{adjoint} satisfies:
\begin{equation}\label{10}
	\begin{alignedat}{2}
		\noalign{\smallskip} 
		K(\varphii,\psi,\zeta) \dis\leq~\!& 
		\widehat{C}(1 + T^2)
		\Biggl(s^{15/2}\lambda^{24}\!\!\jjntQ \!\!{\xi^*}^{15/2}e^{-4s\widehat{\alpha} + 2s\alpha^*} 
		(|\widetilde{\Gvec}|^2 + |\widetilde{g}_1|^2 + |\widetilde{g}_2|^2)                                                                 						\\ 
		\noalign{\smallskip}\dis 
		&\dis+ s^{16}\lambda^{48}\jjntq {\xi^*}^{16}e^{ -8s\widehat{\alpha} + 6s\alpha^*}
		(|\varphii|^2 + |\psi|^2 + |\zeta|^2) \Biggr)
	\end{alignedat}
\end{equation}
	for any $s \geq \widehat{s}(T^4 + T^8)$ and any
\begin{eqnarray*}
	\lambda &\geq &
	\widehat{\lambda}\Bigl( 1 + \|\overline{\yvec}\|_{\infty} + \|\overline{\theta}\|_{\infty} 
	+  \|\overline{c}\|_{\infty} + \|\overline{\Gvec}\|^{1/2}_{\infty} + \|\overline{\Lvec}\|^{1/2}_{\infty} 
	+  \|\overline{g}_1\|^{1/2}_{\infty} + \|\overline{g}_2\|^{1/2}_{\infty}                                                           \\
            &     & 
	+~ \|\overline{l}_1\|^{1/2}_{\infty} + \|\overline{l}_2\|^{1/2}_{\infty} 
	+  \|\overline{\yvec}_t\|^2_{L^2(0,T;\Lvec^{\kappa}(\Om))} + \|\overline{\theta} _t\|^2_{L^2(0,T;L^{\kappa}(\Om))}                 \\
            &     & 
	+~ \|\overline{c} _t\|^2_{L^2(0,T;L^{\kappa}(\Om))} + \exp \left\{ \hat{\lambda}T(1 + \|\overline{\yvec}\|^2_{\infty} 
	+  \|\overline{\theta}\|^2_{\infty} + \|\overline{c}\|^2_{\infty}) \right\} \Bigl).
\end{eqnarray*}
\end{proposition}
%


\section{Proof of Theorem~\ref{teo1}}\label{Sec3}

	Without any lack of generality, we can assume that $N=3$ and $h_1\neq 0$. In order to prove the result, we have to establish 
	some new Carleman estimates. The first one is given in the following result:
\begin{lemma}\label{carlemanteo1}
	Assume that $(\overline{\yvec},\overline{p},\overline{\theta},\overline{c})$ satisfies \eqref{trajectory}--\eqref{kappa}.
	Under the assumptions of Theorem \ref{teo1}, there exist positive constants $C$, $\bar{\alpha}$ and $\tilde{\alpha}$, 
	only depending on $\Om$, $\mathcal{O}$, $T$, $\overline{\yvec}$, $\bar{\theta}$ and $\overline{c}$, satisfying 
	$0<\tilde{\alpha}<\bar{\alpha}$ and $8\tilde{\alpha}-7\bar{\alpha}>0$ such that, for any 
	$(\varphii _T, \psi_T, \zeta_T) \in \Hvec\times L^2(\Omega) \times L^2(\Omega)$ and any 
	$(\widetilde{\Gvec},\widetilde{g}_1,\widetilde{g}_2) \in \Lvec^2(Q)\times L^2(Q)\times L^2(Q)$, the solution to the 
	adjoint system \eqref{adjoint} satisfies:
\begin{equation}
\begin{alignedat}{2}\label{carleman1}
	K(\varphii,\psi,\zeta)\leq   C&\left(\dis\jjntQ e^{-4\tilde{\alpha} + 2\bar{\alpha}\over t^4(T-t)^4}t^{-30}(T-t)^{-30}
	         |\widetilde{\mathbf{G}}|^2)\right.                                                                                            \\ 
	\noalign{\smallskip}\dis
	& ~~+ \dis\jjntQ e^{-16\tilde{\alpha} + 14\bar{\alpha}\over t^4(T-t)^4}t^{-116}(T-t)^{-116}
	         (|\widetilde{g}_1|^2 + |\widetilde{g}_2|^2)                                                                                   \\ 
	\noalign{\smallskip}\dis
	& ~~+ \dis\jjntq e^{-8\tilde{\alpha} + 6\bar{\alpha}\over t^4(T-t)^4}t^{-64}(T-t)^{-64}
	         |\varphi_2|^2                                                                                                                 \\ 
	\noalign{\smallskip}\dis
	& ~~\left.+ \dis\jjntq e^{-16\tilde{\alpha} + 14\bar{\alpha}\over t^4(T-t)^4}t^{-132}(T-t)^{-132}
	       (|\psi|^2+|\zeta|^2)\right).
\end{alignedat}
\end{equation}
\end{lemma}
\begin{proof}
	By choosing $\om\subset\subset\mathcal{O}$, $s_1=\widehat{s}(T^4 + T^8)$,
\begin{eqnarray*}
	\lambda_1 &= &
	\widehat{\lambda}\Bigl( 1 + \|\overline{\yvec}\|_{\infty} + \|\overline{\theta}\|_{\infty} 
	+  \|\overline{c}\|_{\infty} + \|\overline{\Gvec}\|^{1/2}_{\infty} + \|\overline{\Lvec}\|^{1/2}_{\infty} 
	+  \|\overline{g}_1\|^{1/2}_{\infty} + \|\overline{g}_2\|^{1/2}_{\infty}                                                           \\
            &     & 
	+~ \|\overline{l}_1\|^{1/2}_{\infty} + \|\overline{l}_2\|^{1/2}_{\infty} 
	+  \|\overline{\yvec}_t\|^2_{L^2(0,T;\Lvec^{\kappa}(\Om))} + \|\overline{\theta} _t\|^2_{L^2(0,T;L^{\kappa}(\Om))}                 \\
            &     & 
	+~ \|\overline{c} _t\|^2_{L^2(0,T;L^{\kappa}(\Om))} + \exp \left\{ \hat{\lambda}T(1 + \|\overline{\yvec}\|^2_{\infty} 
	+  \|\overline{\theta}\|^2_{\infty} + \|\overline{c}\|^2_{\infty}) \right\} \Bigl),
\end{eqnarray*}
\[
	\overline{\alpha}  = s_1( e^{5/4\lambda_1 m \|\eta^0\|_{\infty}} - e^{\lambda_1 m\|\eta^0\|_{\infty}}),
	\quad 
	\widetilde{\alpha} = s_1( e^{5/4\lambda_1 m \|\eta^0\|_{\infty}} - e^{\lambda_1(m+1)\|\eta^0\|_{\infty}}),
\]
	and 
\begin{equation*}
	C_1 = \widehat{C}(1 + T^2)s_1^{17}\lambda_1^{48} e^{17\lambda_1(m+1)\|\eta^0\|_{\infty}},
\end{equation*}
	we see from~\eqref{10} that
\begin{equation}
\begin{array}{lll}\label{carleman2}
	{\ \ \ \ \ \ \dis       \int\!\!\!\int_Q e^{\frac{-2\overline{\alpha}}{t^4(T-t)^4}}t^4(T-t)^4
	(|\varphii_t|^2+|\psi_{t}|^2+|\zeta_{t}|^2+|\Delta\varphii|^2+|\Delta\psi|^2+|\Delta\zeta|^2)}                                       \\ 
	\noalign{\smallskip}
	{\ \ \ \     \dis     + \int\!\!\!\int_Q e^{\frac{-2\overline{\alpha}}{t^4(T-t)^4}}t^{-4}(T-t)^{-4}
	(|\nabla\varphii|^2+|\nabla\psi|^2+|\nabla\zeta|^2)}                                                                                \\ 
	\noalign{\smallskip}
	{\ \ \ \     \dis     + \int\!\!\!\int_Q e^{\frac{-2\overline{\alpha}}{t^4(T-t)^4}}t^{-12}(T-t)^{-12}
	(|\varphii|^2+|\psi|^2+|\zeta|^2)}                                                                                                  \\ 
	\noalign{\smallskip}
	{            \dis\leq C_1\Biggl(\int\!\!\!\int_Q e^{\frac{-4\widetilde{\alpha}+2\overline{\alpha}}{t^4(T-t)^4}}t^{-30}(T-t)^{-30}
	(|\widetilde{\mathbf{G}}|^2 + |\widetilde{g}_1|^2 + |\widetilde{g}_2|^2)\Biggr.}                                                   \\ 
	\noalign{\smallskip}
	{\ \ \ \     \dis     +  \Biggl.\int\!\!\!\int_{\omega\times(0,T)} e^{\frac{-8\widetilde{\alpha}+6\overline{\alpha}}
	{t^4(T-t)^4}}t^{-64}(T-t)^{-64}(|\varphii|^2 + |\psi|^2 + |\zeta|^2)  \Biggr)}.
\end{array}
\end{equation}

	Notice that $0<\widetilde{\alpha}<\overline{\alpha}$. Moreover, by taking $\lambda_1$ large enough, it can be assumed 
	that $8\widetilde{\alpha}-7\overline{\alpha}>0$.

	Since $h_2\neq 0$, we have
\begin{equation}\label{h}
	|\varphi_1|^2+|\varphi_2|^2+|\varphi_3|^2 \leq C_1 (|\varphi_2|^2+|\varphii\cdot \overrightarrow{\hvec}|^2+|\varphi_3|^2).
\end{equation}
	Thus, by combining \eqref{h} with \eqref{carleman2}, the task is reduced to prove an estimate of the integrals
\begin{equation}\label{B}
	I_3 := \int\!\!\!\int_{\om\times(0,T)}e^{\frac{-8\widetilde{\alpha} + 6\overline{\alpha}}{t^4(T-t)^4}}t^{-64}(T-t)^{-64}|\varphi_3|^2
\end{equation}
	and
\begin{equation}\label{A}
	\ \ \ I_{\hvec} := \int\!\!\!\int_{\omega\times(0,T)}e^{\frac{-8\widetilde{\alpha} + 6\overline{\alpha}}{t^4(T-t)^4}}t^{-64}(T-t)^{-64}
	|\varphii\cdot \overrightarrow{\hvec}|^2
\end{equation}
of the form
$$
	I_3 + I_{\hvec} \leq \varepsilon K(\varphii,\psi,\zeta) + C_{\varepsilon}(\dots),
$$
	where the dots contain local integrals of $\psi$, $\zeta$, $\widetilde{g}_1$ and $\widetilde{g}_2$.

	To do this, let us set
\begin{equation}\label{eq:beta}
	\beta(t) = e^{\frac{-8\widetilde{\alpha}+6\overline{\alpha}}{t^4(T-t)^4}}t^{-64}(T-t)^{-64} \quad \forall t\in(0,T)
\end{equation}
	and let us introduce a cut-off function $\vartheta\in C^2_0(\mathcal{O})$ such that
\begin{equation*}\label{vartheta_0}
	\vartheta\equiv1\hbox{ in }\omega,~0\leq\vartheta\leq1.
\end{equation*}
	For instance, from the differential equation satisfied by $\zeta$, see (\ref{adjoint}), we have:
\begin{equation}\label{65}
\begin{array}{ll}
	\dis      \int\!\!\!\int_{\om\times(0,T)}\beta|\varphii\cdot \overrightarrow{\hvec}|^2 
	&\dis\!\leq ~\int\!\!\!\int_{\mathcal{O}\times(0,T)} \beta\vartheta(\varphii\cdot \overrightarrow{\hvec})
	          (-\zeta_t - \Delta \zeta -\overline{\mathbf{y}} \cdot \nabla\zeta- \widetilde{g}_2)                         \\ 
	\noalign{\smallskip}
	&\dis = ~C \int\!\!\!\int_{\mathcal{O}\times(0,T)}\overrightarrow{\hvec} \cdot 
	[\beta\vartheta\varphii(-\zeta_t  - \Delta \zeta -\overline{\mathbf{y}} \cdot \nabla\zeta-  \widetilde{g}_2)].
\end{array}
\end{equation}
	To get the estimate of $I_{\hvec}$, we will now perform integrations by parts in the last integral and we will ``pass'' 
	all the derivatives from $\zeta$ to $\varphi$:

$\bullet$ 
	First, we integrate by parts with respect to $t$, taking into account that
	\linebreak$\beta(0)=\beta(T)=0$:
\begin{eqnarray}\label{66}
	&& -C\int\!\!\!\int_{\mathcal{O}\times(0,T)} \overrightarrow{\hvec}\cdot(\beta\vartheta\varphii\zeta_t)
	   =C \int\!\!\!\int_{\mathcal{O}\times(0,T)}\overrightarrow{\hvec} 
	   \cdot(\beta_t\vartheta\varphii\zeta+\beta\vartheta\varphii_{t}\zeta)                                    \nonumber\\ 
	\noalign{\smallskip}
	&&\qquad \leq \varepsilon  K(\varphii,\psi,\zeta)																\\ 
	\noalign{\smallskip}
	&&\qquad \ \ +~C(\varepsilon)\int\!\!\!\int_{\mathcal{O}\times(0,T)} e^{\frac{-16\widetilde{\alpha}
	   +14\overline{\alpha}}{t^4(T-t)^4}}t^{-132}(T-t)^{-132}|\zeta|^2.										\nonumber
\end{eqnarray}

$\bullet$ Next, we integrate by parts twice with respect to $x$.
          Here, we use the properties of the cut-off function $\vartheta$:
\begin{eqnarray}\label{67}
	&&-C\int\!\!\!\int_{\mathcal{O}\times(0,T)}\overrightarrow{\hvec} \cdot( \beta\vartheta \varphii \Delta \zeta)  
	= -C\int\!\!\!\int_{\mathcal{O}\times(0,T)}\overrightarrow{\hvec} \cdot[\beta\Delta(\vartheta \varphii) \zeta]               \nonumber \\ 
	\noalign{\smallskip}
	&&\ \ +\ C\int\!\!\!\int_{\partial\mathcal{O}\times(0,T)} \overrightarrow{\hvec} \cdot[\beta\zeta \partial_\nvec(\vartheta \varphii)
	-\beta\vartheta \varphii \partial_\nvec\zeta] \,dS\,dt                                                                                 \\ 
	\noalign{\smallskip}
	&& =\ C\int\!\!\!\int_{\mathcal{O}\times(0,T)}\overrightarrow{\hvec} \cdot\{ \beta[-\Delta\vartheta \varphii-
	2(\nabla\vartheta \cdot\nabla\varphi_1,\nabla\vartheta \cdot\nabla\varphi_2,\nabla\vartheta \cdot\nabla\varphi_3)
	-\vartheta \Delta\varphii]\zeta\}																						   \nonumber\\ 
	\noalign{\smallskip}
	&&\leq\ \varepsilon K(\varphii,\psi,\zeta) + C(\varepsilon)\int\!\!\!\int_{\mathcal{O}\times(0,T)} e^{\frac{-16\widetilde{\alpha}
	+14\overline{\alpha}}{t^4(T-t)^4}}t^{-132}(T-t)^{-132}|\zeta|^2.														   \nonumber
\end{eqnarray}

$\bullet$ We also integrate by parts in the third term with respect to $x$ and we use the incompressibility 
          condition on $\overline{\mathbf{y}}$:
\begin{eqnarray}\label{68}
	&       &- C\int\!\!\!\int_{\mathcal{O}\times(0,T)}\overrightarrow{\hvec} \cdot\ 
	   [\beta\vartheta \varphii(\overline{\mathbf{y}} \cdot \nabla\zeta)]										                \\ 
	\noalign{\smallskip}
	&  =   &- C\int\!\!\!\int_{\mathcal{O}\times(0,T)}\overrightarrow{\hvec} \cdot\ \{ \beta\sum^N_{i=1}\left[\nabla\cdot
	          (\vartheta \varphi_i\overline{\mathbf{y}}\zeta)-[\nabla(\vartheta\varphi_i)\cdot\overline{\mathbf{y}}]
	          \zeta-\vartheta\varphi(\nabla\cdot \overline{\mathbf{y}} )\zeta \right]\mathbf{e}_i \}                      \nonumber\\ 
	\noalign{\smallskip}
	& \leq &\ \varepsilon K(\varphii,\psi,\zeta) + C(\varepsilon)\int\!\!\!\int_{\mathcal{O}\times(0,T)} 
	          e^{\frac{-16\widetilde{\alpha}+14\overline{\alpha}}{t^4(T-t)^4}}t^{-132}(T-t)^{-132}|\zeta|^2.              \nonumber
\end{eqnarray}

$\bullet$ We finally apply Young's inequality in the last term and we get:
\begin{eqnarray}\label{69}
	-\int\!\!\!\int_{\mathcal{O}\times(0,T)}\overrightarrow{\hvec}\cdot( \beta\vartheta \varphii \widetilde{g}_2) 
	&\leq& \varepsilon K(\varphii,\psi,\zeta)																		          \\ 
	\noalign{\smallskip}
	&    & + C(\varepsilon)\int\!\!\!\int_{\mathcal{O}\times(0,T)} e^{\frac{-16\widetilde{\alpha}
	       + 14\overline{\alpha}}{t^4(T-t)^4}}t^{-116}(T-t)^{-116}|\widetilde{g}_2|^2.								\nonumber
\end{eqnarray}

   From \eqref{carleman2} and \eqref{65}-\eqref{69}, by choosing $\varepsilon>0$ sufficiently small, it is easy to deduce 
   the desired estimate of~\eqref{A}. We can argue in the same way starting from \eqref{B} and the equation satisfied by~$\psi$, 
   which leads to a similar estimate.

   Finally, putting all these inequalities together, we find~\eqref{carleman1}.
\end{proof}
\

	We will now deduce a second Carleman inequality with weights that do not vanish at $t = 0$.
	More precisely, let us consider the function
\begin{equation*}
	l(t)=
\left\{
\begin{array}{lll}
	T^2 / 4					&\text{for}& 0  \leq t\leq T/2,\\
	t(T-t)					&\text{for}& T/2\leq t\leq T
\end{array}
\right.
\end{equation*}
	and the following associated weight functions:
\begin{equation*}
\begin{array}{llll}
	\noalign{\smallskip}\dis
	\beta_1=e^{\frac{\overline{\alpha}}{[l(t)]^4}}[l(t)]^{2} \text{,  }\ \ \beta_2(t)= e^{\frac{\overline{\alpha}}{[l(t)]^4}}[l(t)]^{6}   
	        \text{,  }\ \ \beta_3(t)=e^{\frac{2\tilde{\alpha}-\overline{\alpha}}{[l(t)]^4}}[l(t)]^{15} ,                                \\
	\noalign{\smallskip}\dis
	\beta_4(t)=e^{\frac{8\tilde{\alpha}-7\overline{\alpha}}{[l(t)]^4}}[l(t)]^{58} \text{, }\ \ \beta_5(t)=e^{\frac{4\tilde{\alpha}
	-3\overline{\alpha}}{[l(t)]^4}}[l(t)]^{32} \ \ \text{  and  }\ \ \ \beta_6(t)=e^{\frac{8\tilde{\alpha}
	-7\overline{\alpha}}{l(t)[l(t)]^4}}[l(t)]^{66}.
\end{array}
\end{equation*}

	By combining Lemma~\ref{carlemanteo1} and the classical energy estimates satisfied by $\varphi$, $\psi$ and $\zeta$, 
	we easily deduce the following:

\begin{lemma}\label{llema2}
	Assume that $(\overline{\yvec},~\overline{p},~\overline{\theta},~\overline{c})$ satisfies \eqref{trajectory}--\eqref{kappa}.
	Under the assumptions of Theorem~\emph{\ref{teo1}}, there exist positive constants $C$, $\overline{\alpha}$ and~$\tilde{\alpha}$ 
	depending on $\Omega$, $\mathcal{O}$, $T$, $\overline{\yvec}$, $\overline{\theta}$ and $\overline{c}$ and satisfying 
	$0<\tilde{\alpha}<\overline{\alpha}$ and $8\tilde{\alpha}-7\overline{\alpha}>0$ such that, for any 
	$(\varphii _T, \psi_T, \zeta_T) \in \Hvec\times L^2(\Omega) \times L^2(\Omega) $ and any 
	$(\widetilde{\Gvec}, \widetilde{g}_1, \widetilde{g}_2) \in \Lvec^2(Q)\times L^2(Q)\times L^2(Q)$, 
	the solution to the adjoint system \eqref{adjoint} satisfies:
\begin{equation}
\begin{array}{l}\label{ccarleman3}
	\dis
	\int\!\!\!\int_Q \left[\beta_1^{-2}(|\nabla\varphii|^2+|\nabla\psi|^2+|\nabla\zeta|^2)
	+\beta_2^{-2}(|\varphii|^2+|\psi|^2+|\zeta|^2) \right]  \\ 
	\noalign{\smallskip}\dis 
	\quad +\|\varphii(0)\|^2_{\Lvec^2(\Omega)}+\|\psi(0)\|^2_{L^2(\Omega)} +\|\zeta(0)\|^2_{L^2(\Omega)} \\ 
	\noalign{\smallskip}\dis 
	\quad \leq C\Biggl(\int\!\!\!\int_Q\left[\beta_3^{-2}|\widetilde{\Gvec}|^2
	+\beta_4^{-2}(|\widetilde{g}_1|^2+|\widetilde{g}_2|^2)\right] \Biggr.                                    \\ 
	\noalign{\smallskip}\dis 
	\quad \quad\Biggl. +\int\!\!\!\int_{\mathcal{O}\times(0,T)} \left[\beta_5^{-2}|\varphi_2|^2
	+\beta_6^{-2}(|\psi|^2+|\zeta|^2)\right]\Biggr).
\end{array}
\end{equation}
\end{lemma}

	The next step is to prove the null controllability of the linear system \eqref{sislinear}.
	Of course, we will need some specific conditions on the data $\Svec$, $r_1$ and $r_2$.
	Thus, let us introduce the linear operators $M_i$, with
\begin{equation}\label{Ls}
	M_1(\uvec)= \uvec_t  - \Delta \uvec + (\uvec \cdot \nabla) \overline{\yvec} + ( \overline{\yvec}\cdot \nabla) \uvec, 
	\quad
	M_2(\phi)= \phi_t  - \Delta \phi  + \overline{\yvec}\cdot \nabla \phi  
\end{equation}
	and
\begin{equation*}
	M_3(z)= z_t  - \Delta z  +   \overline{\yvec}\cdot \nabla z
\end{equation*}
	and the spaces
\begin{eqnarray*}
	E_0 
		&=& \{\, (\uvec,\phi,z,\vvec,w_1,w_2): \beta_3\uvec \in \Lvec^2(Q),\ \beta_4\phi, \beta_4 z \in L^2(Q),                \\
		& & \ \beta_5\vvec1_{\mathcal{O}} \in \Lvec^2(Q),\  \beta_6w_11_{\mathcal{O}}, \beta_6w_21_{\mathcal{O}} \in L^2(Q),   \\
		& & \ v_1\equiv v_3\equiv0, \ \beta_1^{1/2}\uvec\in L^2(0,T;\Vvec)\cap L^{\infty}(0,T;\Hvec),                  \\
		& & \ \beta_1^{1/2}\phi,\ \beta_1^{1/2}z \in L^2(0,T;H^1_0(\Omega))\cap L^{\infty}(0,T;L^2(\Omega)) \,\}
\end{eqnarray*}
	and
\begin{eqnarray*}
	E_3 
		&=& \{\, (\uvec,q,\phi,z,\vvec,w_1,w_2): (\uvec,\phi,z,\mathbf{v},w_1,w_2) \in E_0,											   \\
		& &\ \beta_1^{1/2}\uvec \in L^4(0,T;\Lvec^{12}(\Omega)), \ q \in L^2_{loc}(Q),   											   \\
		& &\ \beta_1(M_1\uvec + \nabla q -\phi \evec_N - z\overrightarrow{\hvec} - \vvec1_{\mathcal{O}})\in L^2(0,T;\Wvec^{-1,6}(\Omega)),\\
		& &\ \beta_1(M_2\phi +\uvec\cdot  \nabla \overline{\theta}- w_11_{\mathcal{O}}) \in L^2(0,T;H^{-1}(\Omega)),                        \\
		& &\ \beta_1(M_3z+   \uvec\cdot \nabla\overline{c}- w_21_{\mathcal{O}}) \in L^2(0,T;H^{-1}(\Omega)) \,\}.
\end{eqnarray*}

	It is clear that $E_0$ and $E_3$ are Banach spaces for the norms $\|\cdot\|_{E_0}$ and $\|\cdot\|_{E_3}$, where
\begin{eqnarray*}
	 \|(\uvec,\phi,z,\vvec,w_1,w_2)\|_{E_0} & =& \bigl(\|\beta_3\uvec\|^2_{\Lvec^2} +\|\beta_4\phi\|^2_{L^2}+\|\beta_4 z\|^2_{L^2} \\
	& &~~ + \|\beta_5v_21_{\mathcal{O}}\|^2_{L^2} +\|\beta_6w_11_{\mathcal{O}}\|^2_{L^2}
	+\|\beta_6w_21_{\mathcal{O}}\|^2_{L^2}																							 \\
	& &~~ + \|\beta_1^{1/2}\uvec\|^2_{L^2(\Vvec)}+ \|\beta_1^{1/2}\uvec\|^2_{L^{\infty}(\Hvec)}                                 \\
	& &~~ + \|\beta_1^{1/2}\phi\|^2_{L^2(H^{1}_0)}+\|\beta_1^{1/2}\phi||^2_{L^{\infty}(L^2)}                    \\
	& &~~ + \|\beta_1^{1/2}z\|^2_{L^2(H^1_0)}+\|\beta_1^{1/2}z\|^2_{L^{\infty}(L^2)}\bigl)^{1/2}
\end{eqnarray*}
	and
\begin{eqnarray*}
	  \|(\uvec,q,\phi,z,\vvec,w_1,w_2)\|_{E_3} &=& \bigl(\|(\uvec,\phi,z,\vvec,w_1,w_2)\|^2_{E_0}
	   + \|\beta_1^{1/2}\uvec\|^2_{L^4(\Lvec^{12})}																			  \\
	& &~~ +\|\beta_1(M_1\uvec+\nabla q -\phi\evec_N - z\overrightarrow{\hvec}-\vvec1_{\mathcal{O}})\|^2_{L^2(\Wvec^{-1,6})}  \\
	& &~~ +\|\beta_1(M_2\phi +   \uvec\cdot \nabla\overline{\theta}- w_11_{\mathcal{O}})\|^2_{L^2(H^{-1})}                      \\
	& &~~ +\|\beta_1(M_3z+   \uvec\cdot \nabla\overline{c}- w_21_{\mathcal{O}})\|^2_{L^2(H^{-1})}\bigl)^{1/2}.
\end{eqnarray*}

\begin{proposition}\label{controlabilidadenula}
	Assume that $(\overline{\yvec},\overline{p},\overline{\theta},\overline{c})$ satisfies \eqref{trajectory}--\eqref{kappa}.
	Also, assume that $\uvec_0\in \Evec$, $\phi_0,\, z_0\in L^2(\Omega)$ and
\begin{equation*}
	\beta_1(\Svec,r_1,r_2)\in L^2(0,T;\Wvec^{-1,6}(\Omega)) \times L^2(0,T;H^{-1}(\Omega))\times L^2(0,T;H^{-1}(\Omega)).
\end{equation*}
	Then, there exist controls $\vvec$, $w_1$ and $w_2$ such that the associated solution to \eqref{sislinear} belongs to~$E_3$.
	In particular, $v_1\equiv v_3 \equiv 0$, $\uvec(\cdot,T)=\0vec$ and~$\phi(\cdot,T)=z(\cdot,T)=0$.
\end{proposition}

\begin{proof}
	We will follow the general method introduced and used in~\cite{F-I1} for linear parabolic scalar problems.
	
	Thus, let us introduce the auxiliary extremal problem
\begin{equation}\label{controleoptimo}
\left\{
\begin{array}{llr}
	\dis \hbox{Minimize }   \ J(\uvec,q,\phi,z,\vvec,w_1,w_2) 														\\ 
	\noalign{\smallskip}
	\dis \hbox{Subject to } \ \vvec\in \Lvec^2(Q),\ w_1,\ w_2\in L^2(Q), 											\\
	\dis \phantom{\hbox{Subject to }  \ } \text{supp}((\vvec,w_1,w_2)(\cdot,t)) \subset \mathcal{O},~\forall t\in(0,T),
	\\
	\dis \phantom{\hbox{Subject to }  \ }    v_1\equiv v_3\equiv0 \   \text{ and} 																\\ 
	\noalign{\smallskip}
\left\{
\begin{array}{lll}
    M_1(\uvec) + \nabla q = \Svec+ \vvec1_\mathcal{O} + \phi\evec_N + z\overrightarrow{\hvec} 		& \text{in} & Q,     \\
    \nabla \cdot \uvec = 0     																	& \text{in} & Q,     \\
    M_2(\phi)+   \uvec\cdot \nabla\overline{\theta}= r_1+ w_11_\mathcal{O}                                 & \text{in} & Q,     \\
    M_3(z)+   \uvec\cdot \nabla\overline{c}= r_2+ w_21_\mathcal{O}     							    & \text{in} & Q,     \\
    \uvec = \0vec, ~\phi = z = 0           															& \text{on} & \Sigma,\\
   \uvec(\cdot,0) = \uvec_0, \phi(\cdot,0)=\phi_{0}, z(\cdot,0)=z_0 													& \text{in} & \Omega.
\end{array}
\right.
\end{array}
\right.
\end{equation}
	Here, we have used the notation
$$
\begin{array}{l}
	\dis J(\uvec,q,\phi,z,\vvec,w_1,w_2) =   \frac{1}{2}\left\{\int\!\!\!\int_Q \left[ \beta_3^{2}|\uvec|^2
	                                       + \beta_4^{2}(|\phi|^2+|z|^2) \right]  \right.                                                 \\ 
	\noalign{\smallskip}
	\qquad \dis \left. + \int\!\!\!\int_{\mathcal{O}\times(0,T)} \left[ \beta_5^{2}|\vvec|^2+\beta_6^{2}(|w_1|^2+|w_2|^2) \right]\right\}.
\end{array}
$$

	Observe that a solution $(\widehat{\uvec},\widehat{q},\widehat{\phi},\widehat{z},\widehat{\vvec},\widehat{w}_1,\widehat{w}_2)$ to 
	\eqref{controleoptimo} is a good candidate to satisfy $ (\widehat{\uvec},\widehat{q},\widehat{\phi},\widehat{z},\widehat{\vvec},
	\widehat{w}_1,\widehat{w}_2)\in E_3$.

	Let us suppose for the moment that $(\widehat{\uvec},\widehat{q},\widehat{\phi},\widehat{z},\widehat{\vvec},\widehat{w_1},
	\widehat{w_2})$ solves \eqref{controleoptimo}. Then, in view of {\it Lagrange's multipliers theorem,} there must exist dual 
	variables $\widehat{\varphii}$, $\widehat{\pi}$, $\widehat{\psi}$ and $\widehat{\zeta}$ such that
\begin{equation}\label{lagrange}
\left\{
\begin{array}{ll}
	\widehat{\uvec}=\beta_3^{-2}( M_1^*(\widehat{\varphii})+\nabla \widehat{\pi}-\overline{\theta}\nabla\widehat{\psi}
	- \overline{c}\nabla\widehat{\zeta}), \ \ \ \nabla\cdot\widehat{\varphii}=0                          & \text{in } Q, \\ 
	\noalign{\smallskip}
	\widehat{\phi}= \beta_4^{-2}(M^*_2(\widehat{\psi})-\widehat{\varphii}\cdot \evec_N)                    & \text{in } Q, \\
	\noalign{\smallskip}
	\widehat{z}=\beta_4^{-2}( M_3^*(\widehat{\zeta})-\widehat{\varphii}\cdot\overrightarrow{\hvec})        & \text{in } Q, \\
	\noalign{\smallskip}
	\widehat{\vvec}=-\beta_5^{-2}(0,\varphi_2,0),\ \widehat{w}_1=-\beta_6^{-2}\widehat{\psi},
	\ \widehat{w}_2=-\beta_6^{-2}\widehat{\zeta},\ \ \ \ \          								   & \text{in } \mathcal{O}\times(0,T), \\
	\noalign{\smallskip}
	\widehat{\varphii}=\0vec,\ \widehat{\psi}=\widehat{\zeta}=0                                          & \text{on } \Sigma,
\end{array}
\right.
   \end{equation}
	where $M_i^*$ is the adjoint operator of $M_i$ $(i=1,2,3)$, i.e.,
$$
	M_1^*(\varphii)=-\varphii_t  - \Delta\varphii  -D\varphii\overline{\yvec}, \quad
	M_2^*(\psi)=-\psi_t  - \Delta \psi -\overline{\yvec} \cdot \nabla\psi
$$
	and
$$
	M_3^*(\zeta)=-\zeta_t  - \Delta \zeta -\overline{\yvec} \cdot \nabla\zeta.
$$

	Let us introduce the linear space
$$
\begin{array}{l}
	\dis P_0 = \bigg\{\, (\varphii,\pi,\psi,\zeta)\in \Cvec^{2}(\overline{Q}) : 
	\nabla\cdot\varphii=0\text{ in } Q,\  \varphii|_\Sigma = \0vec,       								\\
	\phantom{P_0 = \{\, } \dis \psi|_\Sigma=\zeta|_\Sigma=0,\ \int_{\mathcal{O}}\pi(x,t)\,dx=0 \quad\forall t\in(0,T)\,\bigg\},
\end{array}
$$
	the bilinear form $b(\cdot\,,\cdot) : P_0\times P_0 \to \mathbb{R}$ given by
\begin{eqnarray*}
	& & b((\widehat{\varphii},\widehat{\pi },\widehat{\psi},\widehat{\zeta}),(\varphii,\pi ,\psi,\zeta)) 								\\ 
	\noalign{\smallskip}
	& & \quad \dis := \int\!\!\!\int_Q \beta_3^{-2}( M_1^*(\widehat{\varphii})+\nabla \widehat{\pi }-\overline{\theta}\nabla\widehat{\psi}
	- \overline{c}\nabla\widehat{\zeta}) \cdot( M_1^*(\varphii)+\nabla \pi -\overline{\theta}\nabla\psi- \overline{c}\nabla\zeta)         \\ 
	\noalign{\smallskip}
	& & \quad \dis + \int\!\!\!\int_Q \beta_4^{-2}(M^*_2(\widehat{\psi})-\widehat{\varphii}\cdot \evec_N)(M^*_2(\psi)-\varphii\cdot \evec_N) \\
	\noalign{\smallskip}
	& & \quad \dis + \int\!\!\!\int_Q \beta_4^{-2}( M_3^*(\widehat{\zeta})-\widehat{\varphii}\cdot\overrightarrow{\hvec})
	( M_3^*(\zeta)-\varphii\cdot\overrightarrow{\hvec}) 																				\\ 
	\noalign{\smallskip}
	& & \quad \dis + \int\!\!\!\int_{\mathcal{O}\times(0,T)}(\beta_5^{-2}\widehat{\varphi}_2\cdot\varphi_2
	+\beta_6^{-2}\widehat{\psi}\,\psi+\beta_6^{-2}\widehat{\zeta}\,\zeta)
\end{eqnarray*}
	and the linear form $l_0 : P_0 \to \mathbb{R}$ defined by
\begin{eqnarray*}
	& & \dis \langle l_0,(\varphii,\pi,\psi,\zeta)\rangle := \int^T_0\langle \Svec,\varphii\rangle_{\Hvec^{-1},\mathbf{H}_0^1}\,dt
	+\int^T_0\langle r_1,\psi\rangle_{H^{-1},H_0^1}\,dt\nonumber \\ \noalign{\smallskip}
	& & \quad \dis + \int^T_0\langle r_2,\zeta\rangle_{H^{-1},H_0^1}\,dt+\int_{\Omega}\uvec_0\cdot\varphii(\cdot,0)\,dx \\ \noalign{\smallskip}
	& & \quad \dis + \int_{\Omega}\phi_0\,\psi(\cdot,0)\,dx+\int_{\Omega}z_0\,\zeta(\cdot,0)\,dx\nonumber.
   \end{eqnarray*}
  	Then we must have
   \begin{equation}\label{bilinear}
	b((\widehat{\varphii},\widehat{\pi},\widehat{\psi},\widehat{\zeta}),(\varphii,\pi,\psi,\zeta))
	= \langle l_0,(\varphii,\pi,\psi,\zeta)\rangle \quad \forall(\varphii,\pi,\psi,\zeta)\in P_0,
   	\end{equation}
	i.e.,~the solution to~\eqref{controleoptimo} satisfies \eqref{bilinear}.

   	Conversely, if we are able to solve (\ref{bilinear}) in some sense and then use (\ref{lagrange}) to define $(\widehat{\mathbf{u}},\widehat{q},\widehat{\phi},\widehat{z},\widehat{\mathbf{v}},\widehat{w_1},\widehat{w_2})$, 
	we will have probably found a solution to~(\ref{controleoptimo}).

   	It is clear that $b(\cdot\,,\cdot) : P_0\times P_0 \to \mathbb{R}$ is a symmetric, definite positive and bilinear form on~$P_0$, i.e.,~a scalar product in this linear space.
   	We will denote by $P$ the completion of $P_0$ for the norm induced by $b(\cdot\,,\cdot)$.
   	Then $P$ is a Hilbert space for $b(\cdot\,,\cdot)$.
   	On the other hand, in view of the Carleman estimate~(\ref{ccarleman3}), the linear form $(\varphii,\pi,\psi,\zeta)\mapsto \langle l_0,(\varphii,\pi,\psi,\zeta)\rangle$ is well-defined and continuous on $P$.
   	Hence, from {\it Lax-Milgram's lemma,} we deduce that the variational problem
\begin{equation}\label{depoisdelax}
	\left\{
		\begin{array}{l}
			b((\widehat{\varphii},\widehat{\pi},\widehat{\psi},\widehat{\zeta}),(\varphii,\pi,\psi,\zeta))=\langle l_0,(\varphii,\pi,\psi,\zeta)\rangle,\\ \noalign{\smallskip}
			\forall (\varphii,\pi,\psi,\zeta)\in P, \quad (\widehat{\varphii},\widehat{\pi},\widehat{\psi},\widehat{\zeta})\in P
		\end{array}
	\right.
\end{equation}
	possesses exactly one solution.

   	Let $(\widehat{\varphii},\widehat{\pi},\widehat{\psi},\widehat{\zeta})$ be the unique solution to~\eqref{depoisdelax} and let $\widehat{\mathbf{u}},\widehat{\phi},\widehat{z},\widehat{\mathbf{v}},\widehat{w}_1$ 
	and $\widehat{w}_2$ be given by (\ref{lagrange}). Then, it is readily seen that
\begin{equation}
	\int\!\!\!\int_Q[\beta^2_3|\widehat{\mathbf{u}}|^2+\beta^2_4(|\widehat{\phi}|^2+|\widehat{z}|^2)]+\int\!\!\!\int_{\mathcal{O}\times(0,T)}[\beta_5|\widehat{v}_2|^2+\beta^2_6(|\widehat{w}_1|^2
	+|\widehat{w}_2|^2)]<+\infty
\end{equation}
	and, also, that $(\widehat{\mathbf{u}},\widehat{\phi},\widehat{z})$, together with some $\widehat{q}$, is the unique weak solution
	(belonging to $L^2(0,T;\mathbf{V})\cap L^{\infty}(0,T;\mathbf{H})\times [L^2(0,T;H^1_0(\Omega))\cap L^{\infty}(0,T;L^2(\Omega))]^2$) to the system in 
	(\ref{controleoptimo}) for $\mathbf{v} =\widehat{\mathbf{v}}$, $w_1=\widehat{w}_1$ and $w_2=\widehat{w}_2$.

   	From the arguments in \cite{F-I1}, we also see that $(\widehat{\mathbf{u}},\widehat{q},\widehat{\phi},\widehat{z},\widehat{\mathbf{v}},\widehat{w_1},\widehat{w_2})\in E_3$ 
	and, consequently, the proof is achieved.
\end{proof}

\

   	We can now end the proof of Theorem~\ref{teo1}.

   	We will use the following inverse mapping theorem (see \cite{A-T}):

\begin{theorem} \label{teoremadecontrole}
   	Let $B$ and $G$ be two Banach spaces and let $ \mathcal{A} : B \mapsto G $ satisfying $ \mathcal{A} \in C^1(B;G)$.
   	Assume that $e_0 \in B$, $\mathcal{A}(e_0) = w_0 $ and $\mathcal{A}'(e_0):B \mapsto G $ is surjective.
   	Then there exists $ \delta > 0$ such that, for every $ w \in G$ satisfying $ || w - w_0||_G < \delta,  $ one can find a solution to the equation
$$
\mathcal{A} (e) = w, \ \ e \in B.
$$
\end{theorem}

   	We will apply this result with $B=E_3$, $G=G_1\times G_2$ and
$$
	\mathcal{A}(e) =(\mathcal{A}_1(\mathbf{u},q,\phi,z,\mathbf{v}),\mathcal{A}_2(\mathbf{u},\phi,w_1),\mathcal{A}_3(\mathbf{u},z,w_2), \mathbf{u}(\cdot,0), \phi(\cdot,0),z(\cdot,0))
$$
	for any $e = (\mathbf{u},q,\phi,z,\mathbf{v},w_1,w_2) \in E_3$, where
\begin{equation}
	\begin{array}{l}
		G_1=\beta_1^{-1}[L^2(0,T; \mathbf{W}^{-1,6}(\Omega))\times L^2(0,T; H^{-1}(\Omega))\times L^2(0,T; H^{-1}(\Omega))],\\
		G_2=\mathbf{E}\times L^2(\Omega)\times L^2(\Omega)
	\end{array}
\end{equation}
	and
\begin{equation}
	\begin{array}{l}
		\mathcal{A}_1(\mathbf{u},q,\phi,z,\mathbf{v})=M_1(\mathbf{u}) +(\mathbf{u} \cdot \nabla) \mathbf{u}+ \nabla q - \mathbf{v}1_\mathcal{O}-\phi \mathbf{e}_{N}-z\mathbf{\mathbf{h}},\\
		\mathcal{A}_2(\mathbf{u},\phi,w_1)=M_2(\phi) +   \uvec\cdot \nabla\overline{\theta}+\mathbf{u} \cdot \nabla\phi-  w_11_\mathcal{O}\\
		\mathcal{A}_3(\mathbf{u},z,w_2)=M_3(z)+ \uvec\cdot \nabla\overline{c}+\mathbf{u} \cdot \nabla z-  w_21_\mathcal{O},
	\end{array}
\end{equation}
	for all $(\mathbf{u},q,\phi,z,\mathbf{v},w_1,w_2) \in E_3$.

   	It is not difficult to check that $\mathcal{A}$ is bilinear and satisfies $\mathcal{A}\in C^1(B,G)$.
   	Let $e_0$ be the origin of~$B$. Notice that $\mathcal{A}'(\mathbf{0},0,0,0,\mathbf{0},0,0): B \mapsto G$ is the mapping that, to each $(\mathbf{u},q,\phi,z,\mathbf{v},w_1,w_2)\in B$, 
	associates the function in~$G$ whose components are
$$
	\begin{array}{l}
		\dis M_1(\mathbf{u}) \!+\! \nabla q  \!-\!  \mathbf{v}1_\mathcal{O}  \!-\! \phi \mathbf{e}_{N}  \!-\! z {\mathbf{h}},
		\\ M_2(\phi)  + \uvec\cdot \nabla\overline{\theta} \!-\!  w_11_\mathcal{O},
		\\ M_3(z) +  \uvec\cdot \nabla\overline{c}\!-\! w_21_\mathcal{O}
	\end{array}
$$
	and the initial values $\mathbf{u}(\cdot,0)$, $\phi(\cdot,0)$ and~$z(\cdot,0)$.
   	In view of the null controllability result for (\ref{sislinear}) given in Proposition \ref{controlabilidadenula},  $\mathcal{A}'(\mathbf{0},0,0,0,\mathbf{0},0,0)$ is surjective.

   	Consequently, we can indeed apply Theorem \ref{teoremadecontrole} with these data and, in particular, there exists $\delta >0$ such that, if
$$
	\left\|(\mathbf{0},0,0,\mathbf{u}_0,\phi_{0},z_0)\right\|_G=\left\|(\mathbf{u}_0,\phi_{0},z_0)\right\|_{\mathbf{E}\times L^2(\Omega)\times L^2(\Omega)} \leq \delta,
$$
	we can find controls $\mathbf{v}$, $w_1$ and $w_2$ with ${v}_2\equiv {v}_3\equiv0$ and associated solutions to~\eqref{sisnullcont} that satisfy $\mathbf{u}(\cdot,T) = \mathbf{0}$, $\phi(\cdot,T)= 0$ 
	and $z(\cdot,T)=0$.

   	This ends the proof of Theorem~\ref{teo1}.


\section{Proof of Theorem \ref{teo2}}\label{Sec4}

   	Again, it is not restrictive to assume that $N=3$. We will provisionally impose something stronger than~\eqref{f-case-2}:
\begin{equation}\label{35p}
	\exists \mathbf{k} \in {\mathbb {R}}^3, \ \exists a_0 > 0 \ \hbox{ such that } \ \det\left[ \,\overline{\mathbf{G}} \,|\, \overline{\mathbf{L}} \,|\, \mathbf{k} \,\right] \geq a_0 \ \hbox{ in } \ \mathcal{O}_* \times (0,T).
\end{equation}

 We will need a new different Carleman estimate, which is given in the following lemma:

\begin{lemma}\label{carlemanteo2}
   Assume that $(\overline{\mathbf{y}},\overline{p},\overline{\theta}, \overline{c})$ satisfies \eqref{trajectory}--\eqref{kappa} and~\eqref{35p} holds.
   There exist three positive constants $C$, $\overline{\alpha}$ and~$\tilde{\alpha}$, depending on $\Omega$, $\mathcal{O}$, $T$, $\overline{\mathbf{y}}$, $\overline{\theta}$ and $\overline{c}$ with $0<\tilde{\alpha}<\overline{\alpha}$ and $8\tilde{\alpha}-7\overline{\alpha}>0$ such that, for any $ (\varphii _T, \psi_T, \zeta_T) \in \mathbf{H}\times L^2(\Omega) \times L^2(\Omega) $ and any $(\widetilde{\mathbf{G}}, \widetilde{g}_1, \widetilde{g}_2) \in \mathbf{L}^2(Q)\times L^2(Q)\times L^2(Q)$, the solution to the adjoint system \eqref{adjoint} satisfies:
   \begin{equation}
\begin{array}{lll}\label{carleman3}
K(\varphii,\psi,\zeta)&\leq& C\left(\displaystyle\int\!\!\!\int_Qe^{\frac{-4\tilde{\alpha}+2\overline{\alpha}}{t^4(T-t)^4}}t^{-30}(T-t)^{-30}|\widetilde{\mathbf{G}}|^2)\right.\\ \noalign{\smallskip}
& &+\displaystyle\int\!\!\!\int_Qe^{\frac{-16\tilde{\alpha}+14\overline{\alpha}}{t^4(T-t)^4}}t^{-116}(T-t)^{-116}(|\widetilde{g}_1|^2+|\widetilde{g}_2|^2)\\ \noalign{\smallskip}
& &+\displaystyle\int\!\!\!\int_{\mathcal{O}\times(0,T)}e^{\frac{-8\tilde{\alpha}+6\overline{\alpha}}{t^4(T-t)^4}}t^{-64}(T-t)^{-64}|\varphi_1|^2  \\ \noalign{\smallskip}
& &+\displaystyle\left.
\int\!\!\!\int_{\mathcal{O}\times(0,T)}e^{\frac{-16\tilde{\alpha}+14\overline{\alpha}}{t^4(T-t)^4}}t^{-132}(T-t)^{-132}(|\psi|^2+|\zeta|^2) \right).
\end{array}
   \end{equation}
\end{lemma}

\begin{proof}
   As in the proof of Lemma~\ref{carlemanteo1}, by choosing
   \begin{equation*}
\overline{\alpha}=s_1(e^{5/4\lambda_1 m \|\eta^0\|_{\infty}}-e^{\lambda_1 m\|\eta^0\|_{\infty}}), \quad
\widetilde{\alpha}=s_1(e^{5/4\lambda_1 m \|\eta^0\|_{\infty}}-e^{\lambda_1(m+1)\|\eta^0\|_{\infty}}),
   \end{equation*}
   \begin{equation*}
C_1=\widehat{C}(1+T^2)s_1^{17}\lambda_1^{48}e^{17\lambda_1(m+1)\|\eta^0\|_{\infty}}
   \end{equation*}
and $\omega\subset\subset\mathcal{O}_*$, we see from~\eqref{10} that
   \begin{equation}
\begin{array}{lll}\label{carleman4}
&{\displaystyle\int\!\!\!\int_Qe^{\frac{-2\overline{\alpha}}{t^4(T-t)^4}}t^4(T-t)^4(|\varphii_t|^2+|\psi_{t}|^2+|\zeta_{t}|^2+|\Delta\varphii|^2+
|\Delta\psi|^2+|\Delta\zeta|^2)}\\ \noalign{\smallskip}
&{\displaystyle+\int\!\!\!\int_Qe^{\frac{-2\overline{\alpha}}{t^4(T-t)^4}}t^{-4}(T-t)^{-4}(|\nabla\varphii|^2+|\nabla\psi|^2+|\nabla\zeta|^2)}\\ \noalign{\smallskip}
&{\displaystyle+\int\!\!\!\int_Qe^{\frac{-2\overline{\alpha}}{t^4(T-t)^4}}t^{-12}(T-t)^{-12}(|\varphii|^2+|\psi|^2+|\zeta|^2)}\\ \noalign{\smallskip}
&{\displaystyle\leq C_1\left(\displaystyle\int\!\!\!\int_Qe^{\frac{-4\widetilde{\alpha}+2\overline{\alpha}}{t^4(T-t)^4}}t^{-30}(T-t)^{-30}(|\widetilde{\mathbf{G}}|^2+
|\widetilde{g}_1|^2+|\widetilde{g}_2|^2)\right.}\\ \noalign{\smallskip}
&{\displaystyle+\left.\int\!\!\!\int_{\omega\times(0,T)}e^{\frac{-8\widetilde{\alpha}+6\overline{\alpha}}{t^4(T-t)^4}}t^{-64}(T-t)^{-64}(|\varphii|^2+|\psi|^2+
|\zeta|^2)  \right)}.
\end{array}
   \end{equation}

   Notice that $0<\widetilde{\alpha}<\overline{\alpha}$.
   Moreover, taking $\lambda_1$ large enough, it can be assumed that $8\widetilde{\alpha}-7\overline{\alpha}>0$.

   Recall that $\mathbf{F}$ satisfies (\ref{35p}).
   Let us suppose that, for instance, $\mathbf{k}= \mathbf{e}_1$.
   Then we have:
   \begin{equation}\label{F}
|\varphi_1|^2+|\varphi_2|^2+|\varphi_3|^2\leq C_2 (|\overline{\mathbf{G}}\cdot\varphii|^2+|\overline{\mathbf{L}}\cdot\varphii|^2+|\varphi_1|^2) \ \ \hbox{in} \ \ \mathcal{O}_* \times (0,T)
   \end{equation}
for some $C_2 > 0$.
   Combining (\ref{carleman4}) and (\ref{F}), we thus see that the task is reduced to estimate the integrals
   \begin{equation}
\int\!\!\!\int_{\omega\times(0,T)}e^{\frac{-8\widetilde{\alpha}+6\overline{\alpha}}{t^4(T-t)^4}}t^{-64}(T-t)^{-64}|\overline{\mathbf{G}}\cdot\varphi|^2
   \end{equation}
and
   \begin{equation}
\int\!\!\!\int_{\omega\times(0,T)}e^{\frac{-8\widetilde{\alpha}+6\overline{\alpha}}{t^4(T-t)^4}}t^{-64}(T-t)^{-64}|\overline{\mathbf{L}}\cdot\varphi|^2
   \end{equation}
in terms of $\varepsilon K(\varphi,\psi,\zeta)$ and a constant $C_\varepsilon$ multiplying local integrals of $\psi$, $\zeta$, $\widetilde{g}_1$ and $\widetilde{g}_2$.

   These estimates can be obtained by following the final steps of Lemma~\ref{carlemanteo1};
   as a result, we obtain the inequality~(\ref{carleman3}).
\end{proof}

\

   Let us now give the proof of Theorem~\ref{teo2}.

   First, it is not restrictive to assume that we have \eqref{35p} instead of~\eqref{f-case-2}.
   Indeed, if~\eqref{f-case-2} holds, since $\overline{\mathbf{G}}$ and $\overline{\mathbf{L}}$ are continuous, there exist $\tau, a_0 > 0$, a non-empty open set $\omega \subset \subset \mathcal{O}_*$ and a vector $\mathbf{k} \in \mathbb{R}^3$ such that
   $$
\det\left[ \,\overline{\mathbf{G}} \,|\, \overline{\mathbf{L}} \,|\, \mathbf{k} \,\right] \geq a_0 \ \hbox{ in } \ \overline{\omega} \times [\tau,T-\tau].
   $$
   We can first take $\mathbf{v} \equiv \mathbf{0}$ and $w_1 \equiv w_2 \equiv 0$ for $t \in [0,\tau]$;
   then, we can try to get local exact controllability to $(\overline{\mathbf{y}},\overline{p},\overline{\theta}, \overline{c})$ at time $T-\tau$.
   If appropriate controls are found, they serve to prove Theorem~\ref{teo2}.

   Hence, we can assume that \eqref{35p} is satisfied.
   Arguing as in~Section~\ref{Sec3}, we can deduce from Lemma~\ref{carlemanteo2} the null controllability of the linearized system \eqref{sislinear} with controls like in Theorem~\ref{teo2}
   (that is, an analog of Proposition~\ref{controlabilidadenula});
   then, using again the inverse mapping theorem, we can easily achieve the proof of the desired result.


\section{Proof of Theorem \ref{teo3}}\label{Sec5}

	Without any lack of generality, we can assume that $h_1\neq 0$.
   The proof of our third main result, Theorem~\ref{teo3}, relies on a different and stronger Carleman estimate:

\begin{lemma}\label{carlemanteo3}
   Assume that $N=3$ and $(\overline{\mathbf{y}},\overline{p},\overline{\theta}, \overline{c})$ satisfies \eqref{trajectory}--\eqref{kappa}.
   Under the assumptions of Theorem~\emph{\ref{teo3}}, there exist three positive constants $C$, $\overline{\alpha}$ and~$\tilde{\alpha}$ depending on $\Omega$, $\mathcal{O}$, $T$, $\overline{\mathbf{y}}$, $\overline{\theta}$ and $\overline{c}$ satisfying $0<\tilde{\alpha}<\overline{\alpha}$ and $16\tilde{\alpha}-15\overline{\alpha}>0$ such that, for any $ (\varphi _T, \psi_T, \zeta_T) \in \mathbf{H}\times L^2(\Omega) \times L^2(\Omega) $ and any $(\widetilde{\mathbf{G}}, \widetilde{g}_1, \widetilde{g}_2) \in \mathbf{L}^2(Q)\times L^2(Q)\times L^2(Q)$, the solution to the adjoint system \eqref{adjoint} satisfies:

   \begin{equation}
\begin{array}{lll}\label{carleman5}
K(\varphii,\psi,\zeta)&\leq& C\left(
\displaystyle\int\!\!\!\int_Qe^{\frac{-4\tilde{\alpha}+2\overline{\alpha}}{t^4(T-t)^4}}t^{-30}(T-t)^{-30}
|\widetilde{\mathbf{G}}|^2\right.\\ \noalign{\smallskip}
& &+\displaystyle\int\!\!\!\int_Qe^{\frac{-32\tilde{\alpha}+30\overline{\alpha}}{t^4(T-t)^4}}t^{-252}(T-t)^{-252}
(|\widetilde{g}_1|^2+|\widetilde{g}_2|^2)\\ \noalign{\smallskip}
& &+\displaystyle\left.\int\!\!\!\int_{\mathcal{O}\times(0,T)}e^{\frac{-32\tilde{\alpha}+30\overline{\alpha}}{t^4(T-t)^4}}t^{-268}(T-t)^{-268}(|\psi|^2+|\zeta|^2)\right).
\end{array}
   \end{equation}
\end{lemma}

\begin{proof}
   Again, by choosing
   \begin{equation*}
\overline{\alpha}=s_1(e^{5/4\lambda_1 m \|\eta^0\|_{\infty}}-e^{\lambda_1 m\|\eta^0\|_{\infty}}), \quad
\widetilde{\alpha}=s_1(e^{5/4\lambda_1 m \|\eta^0\|_{\infty}}-e^{\lambda_1(m+1)\|\eta^0\|_{\infty}}),
   \end{equation*}
   \begin{equation*}
C_1=\widehat{C}(1+T^2)s_1^{17}\lambda_1^{48}e^{17\lambda_1(m+1)\|\eta^0\|_{\infty}}
   \end{equation*}
and $\omega\subset \mathcal{O}$, we obtain:
   \begin{equation}
\begin{array}{lll}\label{carleman6}
&{\displaystyle\int\!\!\!\int_Qe^{\frac{-2\overline{\alpha}}{t^4(T-t)^4}}t^4(T-t)^4(|\varphii_t|^2+|\psi_{t}|^2+|\zeta_{t}|^2+|\Delta\varphii|^2+
|\Delta\psi|^2+|\Delta\zeta|^2)}\\ \noalign{\smallskip}
&{\displaystyle\ \ \qquad+\int\!\!\!\int_Qe^{\frac{-2\overline{\alpha}}{t^4(T-t)^4}}t^{-4}(T-t)^{-4}(|\nabla\varphii|^2+|\nabla\psi|^2+|\nabla\zeta|^2)}\\ \noalign{\smallskip}
&{\displaystyle\ \ \qquad +\int\!\!\!\int_Qe^{\frac{-2\overline{\alpha}}{t^4(T-t)^4}}t^{-12}(T-t)^{-12}(|\varphii|^2+|\psi|^2+|\zeta|^2)}\\ \noalign{\smallskip}
&{\displaystyle\ \ \leq C_1\left(\displaystyle\int\!\!\!\int_Qe^{\frac{-4\widetilde{\alpha}+2\overline{\alpha}}{t^4(T-t)^4}}t^{-30}(T-t)^{-30}(|\widetilde{\mathbf{G}}|^2+
|\widetilde{g}_1|^2+|\widetilde{g}_2|^2)\right.}\\ \noalign{\smallskip}
&{\displaystyle\ \ \qquad+\left.\int\!\!\!\int_{\omega\times(0,T)}e^{\frac{-8\widetilde{\alpha}+6\overline{\alpha}}{t^4(T-t)^4}}t^{-64}(T-t)^{-64}(|\varphii|^2+|\psi|^2+ |\zeta|^2)  \right)}.
\end{array}
   \end{equation}

   We notice that $0<\widetilde{\alpha}<\overline{\alpha}$ and, by taking $\lambda_1$ large enough, it can be assumed that $16\widetilde{\alpha}-15\overline{\alpha}>0$.

   Using the incompressibility condition, we get
   \begin{equation}\label{chave}
(-\mathbf{h}_2,\mathbf{h}_1,0)\cdot\nabla \varphi_2= -\partial_1(\mathbf{h}\cdot\varphi)+(\mathbf{h}_3,0,-\mathbf{h}_1)\cdot\nabla\varphi_3.
   \end{equation}
   We will apply (\ref{carleman6}) for the open set $\omega$ defined as follows.
      By assumptions \eqref{controldomain} and \eqref {controldomain-bis}, we choose $\nu>0$ such that
   \begin{center}
$	h_1n_2(\xvec)-h_2n_1(\xvec)\neq 0 \ \ \forall x \in \Gamma_{\nu} := B_{\nu}(x^0)\cap\partial\mathcal{O}\cap\partial\Omega$.
   \end{center}
   
   Then, we introduce
   \begin{equation*}\label{61}
\omega := \{\, \overline{x}\in\Omega:\ \overline{x} = x + \tau (-\mathbf{h}_2,\mathbf{h}_1,0), \ x \in \Gamma_{\nu},\  |\tau|<\tau^0 \,\},
   \end{equation*}
with $\nu,\tau^0>0$  small enough, so that we still have
   \begin{equation*}\label{62}
\omega\subset\mathcal{O}\text{ and } d:= \hbox{dist} (\overline{\omega},\partial\mathcal{O}\cap\Omega)>0.
   \end{equation*}

   Observe that, with this choice, each point $x_*\in\omega$ has the property that one of the point at which the straight line $\{\, x_* + r(-\mathbf{h}_2,\mathbf{h}_1,0) : r \in \mathbb{R} \,\}$ intersects $\partial\Omega$ belongs to $\partial\omega$.

   Once $\omega$ is defined, we apply the inequality (\ref{carleman6}) in this open set and we try to bound the term
   \begin{center}
${\displaystyle\int\!\!\!\int_{\omega\times(0,T)}e^{\frac{-8\widetilde{\alpha}+6\overline{\alpha}}{t^4(T-t)^4}}t^{-64}(T-t)^{-64}|\varphi_2|^2 }$
   \end{center}
in terms of $\varepsilon K(\varphi,\psi,\zeta)$ and local integrals of $\mathbf{h}\cdot\varphi$ and $\varphi_3$.

   To this end, for each $x\in\omega$ we denote by $l(x)$
   (resp.~$\tilde{l}(x)$) the segment that starts from $x$ with direction $(-\mathbf{h}_2,\mathbf{h}_1,0)$ in the positive
   (resp.~negative) sense and ends at $\partial\omega$.
   Then, since $\varphii$ verifies (\ref{chave}) and $\varphii = \mathbf{0}$ on $\Sigma$, it is not difficult to see that
   \begin{equation*}
\varphi_2(x,t)=\int_{l(x)}\left[\partial_1(\mathbf{h}\cdot\varphi)+(\mathbf{h}_3,0,-\mathbf{h}_1)\cdot\nabla\varphi_3\right](\overline{x},t) \,d\overline{x},
   \end{equation*}
   for all $(x,t) \in \omega \times (0,T)$.
   Applying at this point H\"older's inequality and Fubini's formula, we obtain:
   \begin{eqnarray}\label{fubini}
&& \dis \int\!\!\!\int_{\omega\times(0,T)}\beta(t)|\varphi_2|^2  \\ \noalign{\smallskip}
&& \quad \dis \leq C\int_0^T\beta\left(\!\int_{\omega} \int_{l(x)}(
|\partial_1(\mathbf{h}\cdot\varphi)|^2
+|(\mathbf{h}_3,0,-\mathbf{h}_1)\cdot\nabla\varphi_3|^2)\,d\overline{x}\,dx\right)\,dt\nonumber\\ \noalign{\smallskip}
&& \quad \dis = C\int\!\!\!\int_{\omega\times(0,T)}\beta\left[|\partial_1(\mathbf{h}\cdot\varphi)|^2
+|(\mathbf{h}_3,0,-\mathbf{h}_1)\cdot\nabla\varphi_3|^2\right]
\left(\int_{\tilde{l}(\overline{x},t)} \,dx\right) \,d\overline{x}\,dt\nonumber\\ \noalign{\smallskip}
&& \quad \dis \leq C\int\!\!\!\int_{\omega\times(0,T)}\beta\left[|\partial_1(\mathbf{h}\cdot\varphi)|^2+|(\mathbf{h}_3,0,-\mathbf{h}_1)\cdot\nabla\varphi_3|^2\right],\nonumber
   \end{eqnarray}
   recall that $\beta$ is defined in \eqref{eq:beta}.

   Then, let us introduce an appropriate non-empty open set $\omega_{0}$ verifying $\omega\subset\omega_{0}\subset\mathcal{O}$, $d_1:=\hbox{dist}(\overline{\omega_{0}},\partial\mathcal{O}\cap\Omega)>0$ and $d_2:=\hbox{dist}(\overline{\omega},\partial\omega_{0}\cap\Omega)>0$ and a cut-off function $\vartheta_0\in C^2(\overline{\omega}_{0})$ such that
   \begin{center}
$\vartheta_0\equiv1\text{ in }\omega,\ 0\leq\vartheta_0\leq1$ and \\
$\vartheta_0(x)=0$ whenever $x\in\omega_{0}$ and $\hbox{dist}(x,\partial\omega_{0}\cap\Omega)\leq d_2/2$.
   \end{center}
   In particular, $\vartheta_0$ and its derivatives vanish on~$\partial\omega_{0}\cap\Omega$.
   This and the fact that $\varphii = \mathbf{0}$ on~$\Sigma$ imply:
   \begin{eqnarray}\label{A1}
& & \int\!\!\!\int_{\omega\times(0,T)}\beta|\partial_1(\mathbf{h}\cdot\varphi)|^2\leq \int\!\!\!\int_{\omega_0\times(0,T)}\vartheta_0\beta|\partial_1(\mathbf{h}\cdot\varphi)|^2 \nonumber \\ \noalign{\smallskip}
& & \ \ = \int\!\!\!\int_{\omega_0\times(0,T)}\beta\left[\frac{1}{2}\partial_1(\vartheta_0\partial_1|(\mathbf{h}\cdot\varphi)|^2)-\vartheta_0\partial_{11}(\mathbf{h}\cdot\varphi)(\mathbf{h}\cdot\varphi) \right. \\ \noalign{\smallskip}
& & \ \ \qquad\qquad \left. -\frac{1}{2}\partial_1(\partial_1\vartheta_0|(\mathbf{h}\cdot\varphi)|^2)+\frac{1}{2}\partial_{11}\vartheta_0|(\mathbf{h}\cdot\varphi)|^2\right] \nonumber \\ \noalign{\smallskip}
& & \ \ \leq C\int\!\!\!\int_{\omega_0\times(0,T)}\left[\beta|(\mathbf{h}\cdot\varphi)|^2+\beta |\partial_{11}(\mathbf{h}\cdot\varphi)(\mathbf{h}\cdot\varphi)| \right] \nonumber
   \end{eqnarray}
and
   \begin{eqnarray}\label{A2}
& &\int\!\!\!\int_{\omega\times(0,T)}\beta|(\mathbf{h}_3,0,-\mathbf{h}_1)\cdot\nabla\varphi_3|^2\leq C\int\!\!\!\int_{\omega_0\times(0,T)}\vartheta_0\beta|\nabla\varphi_3|^2 \nonumber \\ \noalign{\smallskip}
& & \ \ =C\sum^N_{j=1}\int\!\!\!\int_{\omega_0\times(0,T)}\beta\left[\frac{1}{2}\partial_j(\vartheta_0\partial_j|\varphi_3|^2)-\vartheta_0\partial_{jj}(\varphi_3)\varphi_3 \right. \\  \noalign{\smallskip}
& &  \ \ \qquad\qquad
- \left. \frac{1}{2}\partial_j(\partial_j\vartheta_0|\varphi_3|^2)+\frac{1}{2}\partial_{jj}\vartheta_0|\varphi_3|^2\right]\nonumber \\ \noalign{\smallskip}
&& \ \ \leq C\int\!\!\!\int_{\omega_0\times(0,T)}\left[\beta|\varphi_3|^2+\beta
\Delta\varphi_3\varphi_3\right] . \nonumber
   \end{eqnarray}

   Finally, in view of Young's inequality and classical Sobolev estimates, we see that
   \begin{eqnarray}\label{A1p}
&&\int\!\!\!\int_{\omega\times(0,T)}\beta|\partial_1(\mathbf{h}\cdot\varphi)|^2 \leq C\int\!\!\!\int_{\omega_0\times(0,T)}\beta|(\mathbf{h}\cdot\varphi)|^2\nonumber\\ \noalign{\smallskip}
& & \ \ \qquad +\frac{1}{4C}\int\!\!\!\int_{\omega_0\times(0,T)}
e^{\frac{-2\overline{\alpha}}{t^4(T-t)^4}}t^{4}(T-t)^{4}|\partial_{11}\varphi|^2\nonumber\\ \noalign{\smallskip}
& & \ \ \qquad +C\int\!\!\!\int_{\omega_0\times(0,T)}e^{\frac{-16\widetilde{\alpha} +
14\overline{\alpha}}{t^4(T-t)^4}}t^{-132}(T-t)^{-132}|(\mathbf{h}\cdot\varphi)|^2\\ \noalign{\smallskip}
& & \ \ \leq 2C\int\!\!\!\int_{\omega_0\times(0,T)}e^{\frac{-16\widetilde{\alpha}
+14\overline{\alpha}}{t^4(T-t)^4}}t^{-132}(T-t)^{-132}|(\mathbf{h}\cdot\varphi)|^2\nonumber\\ \noalign{\smallskip}
& & \ \ \qquad + \frac{1}{4C} \int\!\!\!\int_Qe^{\frac{-2\overline{\alpha}}
{t^4(T-t)^4}}t^{4}(T-t)^{4}|\Delta\varphi|^2\nonumber
   \end{eqnarray}
and
   \begin{eqnarray}\label{A2p}
&&\int\!\!\!\int_{\omega\times(0,T)}\beta|(\mathbf{h}_3,0,-\mathbf{h}_1)\cdot\nabla\varphi_3|^2\nonumber\\ \noalign{\smallskip}
& & \ \ \leq\left[ \frac{1}{4C}\int\!\!\!\int_{\omega_0\times(0,T)}
e^{\frac{-2\overline{\alpha}}{t^4(T-t)^4}}t^{4}(T-t)^{4}|\Delta\varphi_3|^2\right.
\nonumber\\ \noalign{\smallskip}
& & \ \ \qquad + C\int\!\!\!\int_{\omega_0\times(0,T)}e^{\frac{-16\widetilde{\alpha}+14\overline{\alpha}}{t^4(T-t)^4}}t^{-132}(T-t)^{-132}|\varphi_3|^2\nonumber\\ \noalign{\smallskip}
& & \ \ \qquad \left.+ C\int\!\!\!\int_{\omega_0\times(0,T)}\beta|\varphi_3|^2\right]\\ \noalign{\smallskip}
& & \ \ \leq 2C\int\!\!\!\int_{\omega_0\times(0,T)}e^{\frac{-16\widetilde{\alpha}
+14\overline{\alpha}}{t^4(T-t)^4}}t^{-132}(T-t)^{-132}|\varphi_3|^2\nonumber\\ \noalign{\smallskip}
& & \ \ \qquad + \frac{1}{4C} \int\!\!\!\int_Qe^{\frac{-2\overline{\alpha}}
{t^4(T-t)^4}}t^{4}(T-t)^{4}|\Delta\varphi|^2\nonumber
   \end{eqnarray}

   Therefore, combining (\ref{carleman6}), (\ref{fubini}), \eqref{h}, (\ref{A1p}) and (\ref{A2p}), we obtain
   \begin{eqnarray}\label{carleman6p}
&&{\displaystyle\int\!\!\!\int_Qe^{\frac{-2\overline{\alpha}}{t^4(T-t)^4}}t^4(T-t)^4(|\varphi_t|^2+|\psi_{t}|^2+|\zeta_{t}|^2+|\Delta\varphi|^2+
|\Delta\psi|^2+|\Delta\zeta|^2)}\nonumber\\ \noalign{\smallskip}
&&{\displaystyle \qquad+\int\!\!\!\int_Qe^{\frac{-2\overline{\alpha}}{t^4(T-t)^4}}t^{-4}(T-t)^{-4}(|\nabla\varphi|^2+|\nabla\psi|^2+|\nabla\zeta|^2)}\nonumber\\ \noalign{\smallskip}
&&{\displaystyle\qquad +\int\!\!\!\int_Qe^{\frac{-2\overline{\alpha}}{t^4(T-t)^4}}t^{-12}(T-t)^{-12}(|\varphi|^2+|\psi|^2+|\zeta|^2)}\\ \noalign{\smallskip}
&&{\displaystyle\leq C\left(\displaystyle\int\!\!\!\int_Qe^{\frac{-4\widetilde{\alpha}+2\overline{\alpha}}{t^4(T-t)^4}}t^{-30}(T-t)^{-30}(|\widetilde{\mathbf{G}}|^2+
|\widetilde{g}_1|^2+|\widetilde{g}_2|^2)\right.}\nonumber\\ \noalign{\smallskip}
&&{\displaystyle\qquad +\int\!\!\!\int_{\omega_0\times(0,T)}e^{\frac{-16\widetilde{\alpha}+14\overline{\alpha}}{t^4(T-t)^4}}t^{-132}(T-t)^{-132}
\left[|(\mathbf{h}\cdot\varphi)|^2+|\varphi_3|^2\right]}\nonumber\\ \noalign{\smallskip}
&&{\displaystyle\qquad \left.+\int\!\!\!\int_{\omega\times(0,T)}e^{\frac{-8\widetilde{\alpha}+6\overline{\alpha}}{t^4(T-t)^4}}t^{-64}(T-t)^{-64}(|\psi|^2+
|\zeta|^2)  \right)}\nonumber.
   \end{eqnarray}

   Once more, our task is reduced to estimate the integrals
   \begin{equation}
\int\!\!\!\int_{\omega_0\times(0,T)}e^{\frac{-16\widetilde{\alpha}+14\overline{\alpha}}{t^4(T-t)^4}}t^{-132}(T-t)^{-132}|(\mathbf{h}\cdot\varphi)|^2
   \end{equation}
and
   \begin{equation}
\int\!\!\!\int_{\omega_0\times(0,T)}e^{\frac{-16\widetilde{\alpha}+14\overline{\alpha}}{t^4(T-t)^4}}t^{-132}(T-t)^{-132}|\varphi_3|^2
   \end{equation}
in terms of $\varepsilon K(\varphi,\psi,\zeta)$ and local integrals of $\psi$, $\zeta$, $\widetilde{g}_1$ and $\widetilde{g}_2$.

   To this end, we can again follow the steps of Lemma~\ref{carlemanteo1};
   after some work, we are led to~(\ref{carleman5}).
\end{proof}


\section{Final comments and questions}\label{Sec6}


\subsection{The case $N=2$}\label{SSec6.1}

   We see from Theorems~\ref{teo1} and~\ref{teo2} that, for $N=2$, even without imposing geometrical hypotheses to $\mathcal{O}$ like (\ref{controldomain}) the local exact controllability to the trajectories holds with two scalar controls $w_1$ and $w_2$.
   In other words, in this case, we only have to act on the PDEs satisfied by $\theta$ and $c$
   (no purely mechanical action is needed).

   A natural question is thus whether Theorem~\ref{teo2} can be improved
   (in the sense that the whole system can be controlled with just one scalar control by imposing (\ref{controldomain}) or any other condition.


\subsection{Nonlinear $\mathbf{F}$ and geometrical conditions on $\mathcal{O}$}\label{SSec6.2}

   In Theorem~\ref{teo3}, we have assumed that $\mathbf{F}$ depends linearly on $\theta$ and $c$.
   This allowed to use the incompressibility condition
   (written in the form \eqref{chave}) and, after several integrations by parts and estimates, led to~\eqref{carleman6p}.

   It is thus reasonable to ask whether a similar result holds for more general functions $\mathbf{F}$ satisfying \eqref{f-case-2} and maybe other conditions.
   But this is to our knowledge an open question.


\subsection{Generalizations to coupled systems with more unknowns}\label{SSec6.3}

   The results in this paper admit several straightforward generalizations.
   For instance, let us assume that $N = 3$.
   With suitable hypotheses, we can obtain a result similar Theorem~\ref{teo2} for the following system in~$Q$
   \begin{equation*}
\left\{
\begin{array}{l}
     \mathbf{y}_t  - \Delta \mathbf{y} +(\mathbf{y} \cdot \nabla) \mathbf{y}+ \nabla p
     = \mathbf{v}1_\mathcal{O}+\mathbf{F}(\theta,c^1,c^2),  \\
     \nabla \cdot \mathbf{y} = 0, \\
     \left(
       \begin{array}{c}
         \theta \\
         c^1 \\
         c^2\\
       \end{array}
     \right)_t-\left(
                 \begin{array}{c}
                          \widetilde{a}\Delta\theta \\
         \widetilde{a}^1\Delta c^1 \\
         \widetilde{a}^2\Delta c^2 \\
                 \end{array}
               \right)+
               \mathbf{y}\cdot\nabla\left(
                 \begin{array}{c}
                          \theta \\
         c^1\\
         c^2 \\
                 \end{array}
               \right)=\left(
                 \begin{array}{c}
                          f(\theta,c^1,c^2) \\
         f^1(\theta,c^1,c^2) \\
         f^2(\theta,c^1,c^2) \\
                 \end{array}
               \right)+
               \left(
                 \begin{array}{c}
                          w1_\mathcal{O} \\
         w^11_\mathcal{O} \\
         w^21_\mathcal{O}\\
                 \end{array}
               \right),
\end{array}
\right.
   \end{equation*}
completed with homogeneous Dirichlet boundary conditions and initial conditions at $t=0$.

   This means that the whole system can be controlled, at least locally, by acting on the PDEs satisfied by $\theta$, $c^1$ and $c^2$, but not on the motion equation.

   Nevertheless, it is unknown whether this can be improved and local controllability can also hold, under some specific assumptions, with at most two scalar controls.


\subsection{Local null controllability without geometrical hypotheses}\label{SSec6.4}

   Let us come back to Theorem~\ref{teo1}.
   Suppose that $(\overline{\mathbf{y}},\overline{p},\overline{\theta}, \overline{c}) \equiv \mathbf{0}$ and let us try to prove a local null controllability result with $L^2$ controls $\mathbf{v} \equiv \mathbf{0}$, $w_1$ and $w_2$, without any assumption on~$\mathcal{O}$.

   Arguing as in~Section~\ref{Sec3}, we readily see that the task is reduced to the proof of a Carleman inequality for the solutions to~\eqref{adjoint} with only local integrals of $\psi$ and $\zeta$ in the right hand side.

   But this inequality is true.
   Indeed, with a self-explained notation, the following holds:

\begin{itemize}

\item [a)]  $\displaystyle \widetilde{I}(s,\lambda;\varphii) \leq C\int\!\!\!\int_{\mathcal{O}\times(0,T)}\rho_1^{-2}(|\varphii\cdot \overrightarrow{\hvec}|^2+|\varphi_3|^2) + \dots$\\
   (from the results in~\cite{Controle sem hipostese sobre O};
   here and below, the dots contain weighted integrals of $|\widetilde{\mathbf{G}}|^2$ and $|\widetilde{g}_1|^2+|\widetilde{g}_2|^2$).

\item [b)]  $\displaystyle K(\psi,\zeta)  \leq  \varepsilon \widetilde{I}(s,\lambda;\varphii) + C \int\!\!\!\int_{\mathcal{O}\times(0,T)}\rho_2^{-2}(|\psi|^2+|\zeta|^2)  + \dots$\\
   (from the usual Carleman estimates for the heat equation).

\end{itemize}

   Using in a) the arguments in the final part of the proof of Lemma~\ref{carlemanteo1}, we obtain an estimate of the form
   $$
\displaystyle \widetilde{I}(s,\lambda;\varphii) \leq \varepsilon \widetilde{I}(s,\lambda;\varphii) + C_\varepsilon \int\!\!\!\int_{\mathcal{O}\times(0,T)}(\rho_3^{-2}|\psi|^2+\rho_4^{-2}|\zeta|^2) + \dots
   $$
   Then, after addition, we find:
   $$
\displaystyle \widetilde{I}(s,\lambda;\varphii) +K(\psi,\zeta) \leq C \int\!\!\!\int_{\mathcal{O}\times(0,T)}\rho^{-2}(|\psi|^2+|\zeta|^2) + \dots,
   $$
which easily leads to the desired estimates.








\section{Appendix}\label{Appendix}

   Let us now present a sketch of the proof of Proposition~\ref{Carleman}.

\

\noindent
$\bullet$ {\sc First estimates:}

\

   In view of the usual Carleman estimates for the heat equation, we easily obtain
   \begin{equation}
   \begin{alignedat}{2}\label{estimate1}
   K(\varphii,\psi,\zeta)\leq&~C\left(\iint_{Q}e^{-2s\alpha}|\nabla\pi|^2dx\,dt\right.\\
   &
   + \iint_{Q}e^{-2s\alpha}(|\mathbf{\widetilde{G}}|^2 + |\widetilde{g}_1|^2  + |\widetilde{g}_2|^2)dx\,dt \\ \noalign{\smallskip}
    &+ \left.s^3\lambda^4\iint_{\mathcal{O}\times(0,T)}e^{-2s\alpha}\xi^3|(\varphii|^2 + | \psi|^2 + | \zeta|^2)\,
    dx\,dt\right)
   \end{alignedat}
   \end{equation}
for all $s\geq s_0(T^7 + T^8)$ and
    \begin{eqnarray*}
\lambda &\geq &\hat{\lambda}\bigl(1+||\overline{\mathbf{y}}||_{\infty} + ||\overline{\theta}||_{\infty}+||\overline{c}||_{\infty}+||\mathbf{\overline{G}}||^{1/2}_{\infty}
+ ||\mathbf{\overline{L}}||^{1/2}_{\infty}\bigr. \\
 & & \bigl. + ||\overline{g}_1||^{1/2}_{\infty}  + ||\overline{g}_2||^{1/2}_{\infty} + ||\overline{l}_1||^{1/2}_{\infty}
+ ||\overline{l}_2||^{1/2}_{\infty} \bigr).
   \end{eqnarray*}

\

\noindent
$\bullet$ {\sc Eliminating the global integral of $\nabla\pi$:}

\

   Let us look at the
   (weak) equation satisfied by the pressure, which can be found by applying the divergence operator to the motion equation of~(\ref{adjoint}):
   \begin{equation}\label{pressure}
\Delta \pi(t) = \nabla\cdot\left[ D\mathbf{\varphii}(t)\overline{\mathbf{y}}(t)+\widetilde{\mathbf{G}}(t)+\overline{\theta}(t)\nabla\psi(t) +  \overline{c}(t)\nabla\zeta(t) \right] \ \ \text{in}\ \Omega, \  t\in(0,T) \ \text{a.e.}
   \end{equation}

   Regarding the right hand side of (\ref{pressure}) like a $H^{-1}$ term, we can apply the main result in~\cite{I-P1} and deduce that
    \begin{eqnarray}\label{25}
     K(\varphii, \psi,\zeta)&\leq&C\left(s^3\lambda^4\int\!\!\!\int_{\mathcal{O}\times(0, T)} e^{-2s\alpha}\xi^3| (\varphii|^2+| \psi|^2 +| \zeta|^2)\nonumber\right.
     \\ \noalign{\smallskip}
     & &+s\int\!\!\!\int_Q e^{-2s\alpha }\xi\left|\widetilde{\mathbf{G}}\right|^2+\int\!\!\!\int_{Q} e^{-2s\alpha }(|\widetilde{g}_1|^2+|\widetilde{g}_2|^2)\nonumber
     \\ \noalign{\smallskip}
     & &+\left.\int\!\!\!\int_{\mathcal{O}_1\times(0,T)} |\hat{\mu}|^2\left|\nabla\pi\right|^2\right)
   \end{eqnarray}
for all $s\geq s_0(T^7 + T^8)$ and all
        \begin{eqnarray*}
\lambda &\geq &\hat{\lambda}\bigl(1+||\overline{\mathbf{y}}||_{\infty} + ||\overline{\theta}||_{\infty}+||\overline{c}||_{\infty}+||\mathbf{\overline{G}}||^{1/2}_{\infty}
+ ||\mathbf{\overline{L}}||^{1/2}_{\infty}\\
 & &+ ||\overline{g}_1||^{1/2}_{\infty}  + ||\overline{g}_2||^{1/2}_{\infty} + ||\overline{l}_1||^{1/2}_{\infty}
+ ||\overline{l}_2||^{1/2}_{\infty} \bigr).
   \end{eqnarray*}

   Taking into account the motion equation in (\ref{adjoint}), we have:
   \begin{eqnarray}\label{26}
& & \int\!\!\!\int_{\mathcal{O}_1\times(0,T)}|\hat{\mu}|^2\left|\nabla\pi\right|^2 \leq C\left(\int\!\!\!\int_{\mathcal{O}_1\times(0,T)} |\hat{\mu}|^2\left|\mathbf{\widetilde{G}}\right|^2\right.\nonumber
      \\ \noalign{\smallskip}
& &  \ \qquad + \int\!\!\!\int_{\mathcal{O}_1\times(0,T)}|\hat{\mu}|^2\left|\varphii_t\right|^2 + \left\|\overline{\mathbf{y}}\right\|^2_{\infty}\int\!\!\!\int_{\mathcal{O}_1\times(0,T)}|\hat{\mu}|^2\left|\nabla\varphii\right|^2
      \\ \noalign{\smallskip}
& & \ \qquad + \left\|\overline{\theta}\right\|^2_{\infty}\int\!\!\!\int_{\mathcal{O}_1\times(0,T)}|\hat{\mu}|^2\left|\nabla\psi\right|^2 + \left\|\overline{c}\right\|^2_{\infty}\int\!\!\!\int_{\mathcal{O}_1\times(0,T)}|\hat{\mu}|^2\left|\nabla \zeta \right|^2\nonumber
      \\ \noalign{\smallskip}
& &\ \qquad + \left.\int\!\!\!\int_{\mathcal{O}_1\times(0,T)}|\hat{\mu}|^2\left|\Delta\varphii\right|^2\right).\nonumber
   \end{eqnarray}
for all $s\geq s_0(T^7 + T^8)$ and all
   \begin{eqnarray*}
\lambda &\geq &\hat{\lambda}\bigl(1+||\overline{\mathbf{y}}||_{\infty} + ||\overline{\theta}||_{\infty}+||\overline{c}||_{\infty}+||\mathbf{\overline{G}}||^{1/2}_{\infty}
+ ||\mathbf{\overline{L}}||^{1/2}_{\infty}\\
 & &+ ||\overline{g}_1||^{1/2}_{\infty}  + ||\overline{g}_2||^{1/2}_{\infty} + ||\overline{l}_1||^{1/2}_{\infty}
+ ||\overline{l}_2||^{1/2}_{\infty} \bigr).
   \end{eqnarray*}

\

\noindent
$\bullet$ {\sc Estimates of the local terms on $\Delta \varphii$ and $\varphii_t$}

\

   The remainder of the proof is devoted to estimate the local terms on $\Delta\varphii$ and $\varphii_t$.
   To do this, we can follow the ideas in~\cite{Guerrero}, which gives

\begin{itemize}

\item [a)] An estimate of $|\Delta\varphii|^2$:
\end{itemize}
   \begin{eqnarray}\label{36}
     \int\!\!\!\int_{\mathcal{O}_1\times(0,T)}|\hat{\mu}|^2|\Delta\varphii|^2&\leq& C(1\!+\! T)\!\left( \displaystyle
     \int\!\!\!\int_{\mathcal{O}_2\times(0,T)}\!\!|\hat{\mu}|^2(|D\varphii\overline{\mathbf{y}}|^2
     \!+\! |\overline{\theta}\nabla\psi|^2 \!+\! |\overline{c}\nabla\zeta|^2 )\right.\nonumber
     \\
     &&\left.+\int\!\!\!\int_{\mathcal{O}_2\times(0,T)}(|\hat{\mu}'\varphii|^2+ |\hat{\mu}\varphii|^2 + |\hat{\mu}\mathbf{\widetilde{G}}|^2) \right);
        \end{eqnarray}
\begin{itemize}

\item  [b)] An estimate of $|\varphii_t|^2$:
\end{itemize}
   \begin{eqnarray}\label{601}
  \int\!\!\!\int_{\mathcal{O}_1\times(0,T)}|\hat\mu|^2|\varphii_t|^2
&\leq&C_\varepsilon \lambda^{24}(1+T)\left( \left\|\mu \mathbf{\widetilde{G}}\right\|^2_{\mathbf{L}^2(Q)}+ \| \mu \widetilde{g}_1 \|^2_{L^2(Q)}+ \| \mu \widetilde{g}_2 \|^2_{L^2(Q)}\right.\nonumber\\
 && +\|\mu \varphii\|^2_{L^2(0,T;\mathbf{L}^2(\mathcal{O}_3))}+|\mu'\varphii|^2_{L^2(0,T;\mathbf{L}^2(\mathcal{O}_3))} \nonumber \\
& & \left. + \|\mu\nabla\varphii\|^2_{L^2(0,T;\mathbf{L}^2(\mathcal{O}_3))}+\|\mu\nabla \psi\|^2_{L^2(0,T;\mathbf{L}^2(\mathcal{O}_3))}\right.\\
& &\left.+\|\mu\nabla \zeta\|^2_{L^2(0,T;\mathbf{L}^2(\mathcal{O}_3))}\right)+\varepsilon K(\varphii,\psi,\zeta)\nonumber,
   \end{eqnarray}
with $\mathcal{O}_1\subset\subset\mathcal{O}_2\subset\subset\mathcal{O}_3\subset\subset\mathcal{O}$.

   Combining (\ref{estimate1}) and (\ref{25})--(\ref{601}), after several additional computations, we find \eqref{10}.

   This ends the proof.


\medskip
\medskip
\end{document}